\documentclass[12pt,oneside]{article}

\usepackage{setspace}

\usepackage{amssymb}   

\usepackage{amsmath}

\usepackage{mathrsfs}

\usepackage{stmaryrd}

\usepackage{amsthm}
\usepackage{newlfont}
\usepackage{amscd}
\usepackage{mathtools}
\usepackage{bm}
\usepackage{tikz-cd}
\usepackage{graphicx}

\usepackage{hyperref}

\usepackage{enumitem}

\usepackage[all]{xy}

\usepackage{textcomp}

\usepackage{amsbsy}

\newtheorem{thm}{Theorem}[section]

\newtheorem{thm-defn}[thm]{Theorem/Definition}
\newtheorem{lem}[thm]{Lemma}
\newtheorem{prop}[thm]{Proposition}
\newtheorem{cor}[thm]{Corollary}

\theoremstyle{definition}
\newtheorem{defn}[thm]{Definition}

\theoremstyle{remark}
\newtheorem{rem}[thm]{Remark}

\numberwithin{equation}{section}

\begin{document}

\pagenumbering{arabic}

\title{Relative crystalline representations and $p$-divisible groups in the small ramification case}

\author{Tong Liu and Yong Suk Moon}
\date{}

\maketitle

\begin{abstract}
Let $k$ be a perfect field of characteristic $p > 2$, and let $K$ be a finite totally ramified extension over $W(k)[\frac{1}{p}]$	 of ramification degree $e$. Let $R_0$ be a relative base ring over $W(k)\langle t_1^{\pm 1}, \ldots, t_m^{\pm 1}\rangle$ satisfying some mild conditions, and let $R = R_0\otimes_{W(k)}\mathcal{O}_K$. We show that if $e < p-1$, then every crystalline representation of $\pi_1^{\text{\'et}}(\mathrm{Spec}R[\frac{1}{p}])$ with Hodge-Tate weights in $[0, 1]$ arises from a $p$-divisible group over $R$. 
\end{abstract}

\section{Introduction} \label{sec:1}

Let $k$ be a perfect field of characteristic $p > 2$, and let $W(k)$ be its ring of Witt vectors. Let $K$ be a finite totally ramified extension over $W(k)[\frac{1}{p}]$ with ramification degree $e$, and denote by $\mathcal{O}_K$ its ring of integers. If $G$ is a $p$-divisible group over $\mathcal{O}_K$, then it is well-known that its Tate module $T_p(G)$ is a crystalline $\mathrm{Gal}(\overline{K}/K)$-representation with Hodge-Tate weights in $[0, 1]$. Conversely, Kisin showed the following result in \cite{kisin-crystalline}.

\begin{thm} \label{thm:1.1} \emph{(cf. \cite[Corollary 2.2.6]{kisin-crystalline})}
Let $T$ be a crystalline $\mathrm{Gal}(\overline{K}/K)$-representation finite free over $\mathbf{Z}_p$ whose Hodge-Tate weights lie in $[0, 1]$. Then there exists a $p$-divisible group $G$ over $\mathcal{O}_K$ such that $T_p(G) \cong T$ as $\mathrm{Gal}(\overline{K}/K)$-representations.	
\end{thm}

\noindent The goal of this paper is to study the analogous statement in the relative case.

Recently, relative $p$-adic Hodge theory generalizing the classical theory has been developed by Brinon \cite{brinon-relative}, Scholze \cite{scholze-p-adic-hodge}, Kedlaya-Liu \cite{kedlaya-liu-relative-padichodge}, and Diao-Lan-Liu-Zhu \cite{diao-lan-liu-zhu-logRH}. It aims to understand certain $p$-adic \'etale local systems, namely \textit{de Rham} local systems on a smooth (rigid analytic) variety $X$ over a $p$-adic field. As in the classical case, one can expect that studying the full subcategory of \textit{crystalline} local systems on $X$, when it is well-defined, would exhibit a lot of useful information.

Let $R_0$ be a base ring over $W(k)\langle t_1^{\pm 1}, \ldots, t_m^{\pm 1}\rangle$ given as in Section \ref{sec:2.1}, and let $R = R_0\otimes_{W(k)}\mathcal{O}_K$. In this paper, we will work in the setting $X = \mathrm{Spec} (R[\frac{1}{p}])$. For this case, the category of $p$-adic \'etale local systems on $X$ is equivalent to that of $\mathbf{Q}_p$-representations of the \'etale fundamental group $\mathcal{G}_R$ of $X$ as $X$ is connected. Moreover, for representations of $\mathcal{G}_R$, the condition being \textit{crystalline} is well-defined by \cite{brinon-relative}. If $G_R$ is a $p$-divisible group over $R$, its Tate module $T_p(G_R)$ is a crystalline $\mathcal{G}_R$-representation with Hodge-Tate weights in $[0, 1]$ (cf. \cite{kim-groupscheme-relative}). A natural question is whether the converse statement analogous to Theorem \ref{thm:1.1} will hold in the relative case. We prove that the answer is affirmative if the ramification index $e$ is small.

\begin{thm} \label{thm:1.2}
Suppose $e < p-1$. Let $T$ be a crystalline $\mathcal{G}_R$-representation finite free over $\mathbf{Z}_p$ whose Hodge-Tate weights lie in $[0, 1]$. Then there exists a $p$-divisible group $G_R$ over $R$ such that $T_p(G_R) \cong T$ as $\mathcal{G}_R$-representations.
\end{thm}

As an immediate corollary using the results in \cite{moon-relativeRaynaud}, when $R$ has Krull dimension $2$, we obtain the following result on the geometry of the locus of crystalline $\mathcal{G}_R$-representations with Hodge-Tate weights in $[0, 1]$. For a fixed absolutely irreducible $\mathbf{F}_p$-representation $V_0$ of $\mathcal{G}_R$, there exists a universal deformation ring which parametrizes the deformations of $V_0$ (\cite{smit}). By \cite[Theorem 5.7]{moon-relativeRaynaud}, we deduce:

\begin{cor} \label{cor:1.3}
Suppose $R$ has Krull dimension $2$ and $e < p-1$. Then the locus of crystalline representations of $\mathcal{G}_R$ with Hodge-Tate weights in $[0, 1]$ cuts out a closed subscheme of the universal deformation scheme. 	
\end{cor}

\noindent We give a more precise statement of Corollary \ref{cor:1.3} in Section \ref{sec:6}.

There are three major ingredients for the proof of Theorem \ref{thm:1.2}. Firstly, Brinon and Trihan proved in \cite{brinon-trihan} the generalization of Theorem \ref{thm:1.1} for the case when the base is a complete discrete valuation ring whose residue field has a finite $p$-basis. We use this result together with the fact that the $p$-adic completion of $R_{0, (p)}$ is an example of such rings studied \textit{loc. cit}. Secondly, Kim generalized the Breuil-Kisin classification in the relative setting in \cite{kisin-crystalline}, and showed that the category of $p$-divisible groups over $R$ is anti-equivalent to the category of Kisin modules of height $1$ over $R_0[\![u]\!]$. Using the classification, we reduce our problem to constructing desired Kisin modules. We remark that our method of constructing the appropriate Kisin modules relies on the assumption that $e < p-1$. Lastly, to show the statement for the special case when $R$ is a formal power series ring of dimension $2$, we use the purity result for $p$-divisible groups proved in \cite{vasiu-zink-purity} when the ramification index is small.  

\subsection{Notations} We will reserve $\varphi$ for various Frobenius. To be more precise,  let $A$ be an $W(k)$-algebra on which the arithmetic Frobenius $\varphi$ on $W(k)$ extends, and $M$ an $A$-module. We denote $\varphi_A: A \to A$ for such an extension. Let $\varphi_M: M \to M$ be a $\varphi_A$-semi-linear map. This is equivalent to having an $A$-linear map $1 \otimes \varphi _M: \varphi_A^*M \to M$, where $\varphi_A ^* M $ denotes $ A \otimes_{\varphi_A, A } M$. We always drop the subscripts $A$ and $M$ from $\varphi$ if no confusion arises.  Let $f: A\to B$ be a ring map compatible with Frobenius, that is, $f\circ \varphi_A= \varphi_B\circ f$. Then $\varphi_M$ naturally extends to $\varphi_{M_B}: M_B \to M_B$ for $M_B:= B \otimes_A M$. It is easy to check that $\varphi ^*_B M_B = B \otimes_{A} \varphi ^*_A M$ and $1 \otimes \varphi_{M_B} : \varphi^*_B M_B \to M_B$ is equal to $B \otimes _A (1 \otimes \varphi_M)$.

\section{Relative $p$-adic Hodge theory and \'etale $\varphi$-modules} \label{sec:2}

\subsection{Base ring and crystalline period ring in the relative case} \label{sec:2.1}

We follow the same notations as in the Introduction. We recall the assumptions on the base rings and the construction of crystalline period ring in relative $p$-adic Hodge theory in \cite{brinon-relative} (and in \cite{kim-groupscheme-relative} for Breuil-Kisin classification). We also impose some mild additional assumptions which will be needed later. Let $R_0$ be a ring obtained from $W(k)\langle t_1^{\pm 1}, \ldots, t_m^{\pm 1}\rangle$ by a finite number of iterations of the following operations:
\begin{itemize}
\item $p$-adic completion of an \'etale extension;
\item $p$-adic completion of a localization;
\item completion with respect to an ideal containing $p$.	
\end{itemize}

\noindent We assume that either $W(k)\langle t_1^{\pm 1}, \ldots, t_m^{\pm 1}\rangle \rightarrow R_0$ has geometrically regular fibers or $R_0$ has Krull dimension less than $2$, and that $k \rightarrow R_0/pR_0$ is geometrically integral. In addition, we suppose that $R_0$ is an integral domain containing a Cohen ring $W$ over $W(k)$ such that $R_0$ is formally finite type over $W$, and that $R_0/pR_0$ is a unique factorization domain. 

$R_0/pR_0$ has a finite $p$-basis given by $\{t_1, \ldots, t_m\}$ in the sense of \cite[Definition 1.1.1]{deJong-dieudonnemodule}. The Witt vector Frobenius on $W(k)$ extends (not necessarily uniquely) to $R_0$, and we fix such a Frobenius endomorphism $\varphi: R_0 \rightarrow R_0$. Let $\widehat{\Omega}_{R_0} \coloneqq \varprojlim_n \Omega_{(R_0/p^n)/W(k)}$ be the module of $p$-adically continuous K\"{a}hler differentials. By \cite[Proposition 2.0.2]{brinon-relative}, $\widehat{\Omega}_{R_0} \cong \bigoplus_{i=1}^m R_0\cdot dt_i$. We work over the base ring $R$ given by  $R \coloneqq R_0\otimes_{W(k)}\mathcal{O}_K$.

Let $\overline{R}$ denote the union of finite $R$-subalgebras $R'$ of a fixed separable closure of $\mathrm{Frac}(R)$ such that $R'[\frac{1}{p}]$ is \'etale over $R[\frac{1}{p}]$. Then $\mathrm{Spec} \overline{R}[\frac{1}{p}]$ is a pro-universal covering of $\mathrm{Spec}R[\frac{1}{p}]$, and $\overline{R}$ is the integral closure of $R$ in $\overline{R}[\frac{1}{p}]$. Let $\mathcal{G}_R \coloneqq \mathrm{Gal}(\overline{R}[\frac{1}{p}]/R[\frac{1}{p}]) = \pi_1^{\text{\'et}}(\mathrm{Spec}R[\frac{1}{p}])$. By a representation of $\mathcal{G}_R$, we always mean a finite continuous representation.

The crystalline period ring $B_{\mathrm{cris}}(R)$ is constructed as follows. Let $\displaystyle \overline{R}^{\flat} = \varprojlim_{\varphi} \overline{R}/p\overline{R}$. There exists a natural $W(k)$-linear surjective map $\theta: W(\overline{R}^{\flat}) \rightarrow \widehat{\overline{R}}$ which lifts the projection onto the first factor. Here, $\widehat{\overline{R}}$ denotes the $p$-adic completion of $\overline{R}$. Let $\theta_{R_0}: R_0\otimes_{W(k)}W(\overline{R}^{\flat}) \rightarrow \widehat{\overline{R}}$ be the $R_0$-linear extension of $\theta$. Define the integral crystalline period ring $A_{\mathrm{cris}}(R)$ to be the $p$-adic completion of the divided power envelope of $R_0\otimes_{W(k)}W(\overline{R}^{\flat})$ with respect to $\mathrm{ker}(\theta_{R_0})$. Choose compatibly $\epsilon_n \in \overline{R}$ such that $\epsilon_0 =1, ~\epsilon_n = \epsilon_{n+1}^p$ with $\epsilon_1 \neq 1$, and let $\widetilde{\epsilon} = (\epsilon_n)_{n \geq 0} \in \overline{R}^{\flat}$. Then $\tau \coloneqq \log{[\widetilde{\epsilon}]}\in A_{\mathrm{cris}}(R)$. Define $B_{\mathrm{cris}}(R) = A_{\mathrm{cris}}(R)[\frac{1}{\tau}]$. $B_{\mathrm{cris}}(R)$ is equipped naturally with $\mathcal{G}_R$-action and Frobenius endomorphism, and $B_{\mathrm{cris}}(R)\otimes_{R_0[\frac{1}{p}]}R[\frac{1}{p}]$ is equipped with a natural filtration by $R[\frac{1}{p}]$-submodules. Furthermore, we have a natural integrable connection $\nabla: B_{\mathrm{cris}}(R) \rightarrow B_{\mathrm{cris}}(R)\otimes_{R_0}\widehat{\Omega}_{R_0}$ such that Frobenius is horizontal and Griffiths transversality is satisfied.  

For a $\mathcal{G}_R$-representation $V$ over $\mathbf{Q}_p$, let $D_{\mathrm{cris}}(V) \coloneqq \mathrm{Hom}_{\mathcal{G}_R}(V, B_{\mathrm{cris}}(R))$. The natural morphism
\[
\alpha_{\mathrm{cris}}:  D_{\mathrm{cris}}(V)\otimes_{R_0[\frac{1}{p}]} B_{\mathrm{cris}}(R)  \rightarrow V^{\vee}\otimes_{\mathbf{Q}_p}  B_{\mathrm{cris}}(R)
\] 
is injective. We say $V$ is \textit{crystalline} if $\alpha_{\mathrm{cris}}$ is an isomorphism. When $V$ is crystalline, then $D_{\mathrm{cris}}(V)$ is a finite projective $R_0[\frac{1}{p}]$-module, and $D_{\mathrm{cris}}(V)\otimes_{R_0[\frac{1}{p}]} R[\frac{1}{p}]$ has the filtration induced by that on $B_{\mathrm{cris}}(R)\otimes_{R_0[\frac{1}{p}]}R[\frac{1}{p}]$. We define the Hodge-Tate weights similarly as in the classical $p$-adic Hodge theory. Frobenius and connection on $B_{\mathrm{cris}}(R)$ induce those structures on $D_{\mathrm{cris}}(V)$; for the Frobenius endomorphism on $D_{\mathrm{cris}}(V)$, $1\otimes \varphi: \varphi^*D_{\mathrm{cris}}(V) \rightarrow D_{\mathrm{cris}}(V)$ is an isomorphism, and the connection $\nabla: D_{\mathrm{cris}}(V) \rightarrow D_{\mathrm{cris}}(V)\otimes_{R_0}\widehat{\Omega}_{R_0}$ is integrable and topologically quasi-nilpotent. Furthermore, Griffiths transversality is satisfied and $\varphi$ is horizontal. For a $\mathcal{G}_R$-representation $T$ which is free over $\mathbf{Z}_p$, we say it is crystalline if $T[\frac{1}{p}]$ is crystalline. 

Suppose $S_0$ is another relative base ring over $W(k)\langle t_1^{\pm 1}, \ldots, t_m^{\pm 1}\rangle$ satisfying the above conditions and equipped with a choice of Frobenius, and let $b: R_0 \rightarrow S_0$ be a $\varphi$-equivariant $W(k)\langle t_1^{\pm 1}, \ldots, t_m^{\pm 1}\rangle$-algebra map. We also denote $b: R = R_0\otimes_{W(k)}\mathcal{O}_K \rightarrow S \coloneqq S_0\otimes_{W(k)}\mathcal{O}_K$ the map induced $\mathcal{O}_K$-linearly. By choosing a common geometric point, this induces a map of Galois groups $\mathcal{G}_S \rightarrow \mathcal{G}_R$, and also a map of crystalline period rings $B_{\mathrm{cris}}(R) \rightarrow B_{\mathrm{cris}}(S)$ compatible with all structures. If $V$ is a crystalline representation of $\mathcal{G}_R$ with certain Hodge-Tate weights, then via these maps $V$ is also a crystalline representation of $\mathcal{G}_S$ with the same Hodge-Tate weights, and the construction of $D_{\mathrm{cris}}(V)$ is compatible with the base change. 

We will consider the following base change maps in later sections. Let $\mathcal{O}_{L_0}$ be the $p$-adic completion of $R_{0, (p)}$, and let $b_L: R_0 \rightarrow \mathcal{O}_{L_0}$ be the natural $\varphi$-equivariant map. This induces $b_L: R \rightarrow \mathcal{O}_L \coloneqq \mathcal{O}_{L_0}\otimes_{W(k)}\mathcal{O}_K$. Note that $L  = \mathcal{O}_L[\frac{1}{p}]$ is an example of a complete discrete valuation field with a residue field having a finite $p$-basis, studied in \cite{brinon-trihan}. On the other hand, for each maximal ideal $\mathfrak{q} \in \mathrm{mSpec} R_0$, let $\widehat{R_{0, \mathfrak{q}}}$ be the $\mathfrak{q}$-adic completion of $R_{0, \mathfrak{q}}$. By the structure theorem of complete regular local rings, we have $\widehat{R_{0, \mathfrak{q}}} \cong \mathcal{O}_{\mathfrak{q}}[\![s_1, \ldots, s_l]\!]$ where $\mathcal{O}_{\mathfrak{q}}$ is a Cohen ring with the maximal ideal $(p)$ and $l \geq 0$ is an integer ($\widehat{R_{0, \mathfrak{q}}}$ is understood to be $\mathcal{O}_{\mathfrak{q}}$ when $l = 0$). We consider the natural $\varphi$-equivariant morphism $b_{\mathfrak{q}}: R_0 \rightarrow \widehat{R_{0, \mathfrak{q}}}$, which induces $b_{\mathfrak{q}}: R \rightarrow R_{\mathfrak{q}} \coloneqq \widehat{R_{0, \mathfrak{q}}}\otimes_{W(k)}\mathcal{O}_K$.

\subsection{\'Etale $\varphi$-modules} \label{sec:2.2}

We study \'etale $\varphi$-modules and associated Galois representations. Most of the material in this section is a review of \cite{kim-groupscheme-relative} Section 7, and the underlying geometry is based on perfectoid spaces as in \cite{scholze-perfectoid}. 

Let $R_0$ be a relative base ring over $W(k)\langle t_1^{\pm 1}, \ldots, t_m^{\pm 1}\rangle$ and let $R = R_0\otimes_{W(k)}\mathcal{O}_K$ as above. Choose a uniformizer $\varpi \in \mathcal{O}_K$. For integers $n \geq 0$, we choose compatibly $\varpi_n \in \overline{K}$ such that $\varpi_0 = \varpi$ and $\varpi_{n+1}^p = \varpi_n$, and let $K_{\infty}$ be the $p$-adic completion of $\bigcup_{n \geq 0} K(\varpi_n)$. Then $K_{\infty}$ is a perfectoid field and $(\widehat{\overline{R}}[\frac{1}{p}], \widehat{\overline{R}})$ is a perfectoid affinoid $K_{\infty}$-algebra. Let $K_{\infty}^\flat$ denote the tilt of $K_{\infty}$ as defined in \cite{scholze-perfectoid}, and let $\underline{\varpi} \coloneqq (\varpi_n) \in K_{\infty}^\flat$. 

Let $\mathfrak{S} \coloneqq R_0[\![u]\!]$ equipped with the Frobenius extending that on $R_0$ by $\varphi(u) = u^p$. Let $E_{R_\infty}^+ = \mathfrak{S}/p\mathfrak{S}$, and let $\tilde{E}_{R_\infty}^+$ be the $u$-adic completion of $\varinjlim_{\varphi}E_{R_\infty}^+$. Let $E_{R_\infty} = E_{R_\infty}^+[\frac{1}{u}]$ and $\tilde{E}_{R_\infty} = \tilde{E}_{R_\infty}^+[\frac{1}{u}]$. By \cite[Proposition 5.9]{scholze-perfectoid}, $(\tilde{E}_{R_\infty}, \tilde{E}_{R_\infty}^+)$ is a perfectoid affinoid $L^\flat$-algebra, and we have the natural injective map $(\tilde{E}_{R_\infty}, \tilde{E}_{R_\infty}^+) \hookrightarrow (\overline{R}^\flat[\frac{1}{\underline{\varpi}}], \overline{R}^\flat)$ given by $u \mapsto \underline{\varpi}$. 

Let 
\begin{equation} \label{eq:2.1}
\tilde{R}_{\infty} \coloneqq W(\tilde{E}_{R_\infty}^+)\otimes_{W(K_{\infty}^{\flat \circ}), ~\theta} \mathcal{O}_{K_{\infty}}.	
\end{equation}
By \cite[Remark 5.19]{scholze-perfectoid}, $(\tilde{R}_\infty[\frac{1}{p}], \tilde{R}_\infty)$ is a perfectoid affinoid $K_{\infty}$-algebra whose tilt is $(\tilde{E}_{R_\infty}, \tilde{E}_{R_\infty}^+)$. Furthermore, it is shown in \cite{kim-groupscheme-relative} that we have a natural injective map $(\tilde{R}_\infty[\frac{1}{p}], \tilde{R}_\infty) \hookrightarrow (\widehat{\overline{R}}[\frac{1}{p}], \widehat{\overline{R}})$ whose tilt is $(\tilde{E}_{R_\infty}, \tilde{E}_{R_\infty}^+) \hookrightarrow (\overline{R}^\flat[\frac{1}{\underline{\varpi}}], \overline{R}^\flat)$. For $\mathcal{G}_{\tilde{R}_\infty} \coloneqq \pi_1^{\text{\'{e}t}}(\mathrm{Spec}\tilde{R}_\infty[\frac{1}{p}])$, we then have a continuous map of Galois groups $\mathcal{G}_{\tilde{R}_\infty} \rightarrow \mathcal{G}_R$, which is a closed embedding by \cite[Proposition 5.4.54]{gabber-almost}. By the almost purity theorem in \cite{scholze-perfectoid}, $\overline{R}^\flat[\frac{1}{\underline{\varpi}}]$ can be canonically identified with the $\underline{\varpi}$-adic completion of the affine ring of a pro-universal covering of $\mathrm{Spec}\tilde{E}_{R_\infty}$, and letting $\mathcal{G}_{\tilde{E}_{R_\infty}}$ be the Galois group corresponding to the pro-universal covering, there exists a canonical isomorphism $\mathcal{G}_{\tilde{E}_{R_\infty}} \cong \mathcal{G}_{\tilde{R}_\infty}$. 

\begin{lem} \label{lem:2.1}
Consider the map of Galois groups $\mathcal{G}_{\mathcal{O}_L} \rightarrow \mathcal{G}_R$ induced by choosing a common geometric point for the base change map $b_L: R \rightarrow \mathcal{O}_L$ in Section \ref{sec:2.1}. Then the images of $\mathcal{G}_{\mathcal{O}_L}$ and $\mathcal{G}_{\tilde{R}_\infty}$ inside $\mathcal{G}_R$ generate the group $\mathcal{G}_R$.	
\end{lem}

\begin{proof}
$E_{R_\infty}^+$ has a finite $p$-basis given by $\{t_1, \ldots, t_m, u\}$. Note that for any element of $g \in \mathcal{G}_R$, there exists an element $h \in \mathcal{G}_{\mathcal{O}_L}$ whose image in $\mathcal{G}_R$ induces the same actions on $t_1^{\frac{1}{p^{\infty}}}, \ldots, t_m^{\frac{1}{p^{\infty}}}, \varpi^{\frac{1}{p^{\infty}}}$. Since $\tilde{R}_{\infty} = W(\tilde{E}_{R_\infty}^+)\otimes_{W(K_{\infty}^{\flat \circ}), ~\theta} \mathcal{O}_{K_{\infty}}$, the actions of $g$ and $h$ are the same on the elements of 	$\tilde{R}_{\infty}$. Hence, the assertion follows. 
\end{proof}

Now, let $\mathcal{O}_{\mathcal{E}}$ be the $p$-adic completion of $\mathfrak{S}[\frac{1}{u}]$. Note that $\varphi$ on $\mathfrak{S}$ extends naturally to $\mathcal{O}_{\mathcal{E}}$.

\begin{defn} \label{defn:2.2}
An \textit{\'{e}tale} $(\varphi, \mathcal{O}_\mathcal{E})$-\textit{module} is a pair $(\mathcal{M}, \varphi_{\mathcal{M}})$ where $\mathcal{M}$ is a finitely generated $\mathcal{O}_\mathcal{E}$-module and $\varphi_{\mathcal{M}}: \mathcal{M} \rightarrow \mathcal{M}$ is a $\varphi$-semilinear endomorphism such that $1\otimes\varphi_{\mathcal{M}}: \varphi^*\mathcal{M} \rightarrow \mathcal{M}$ is an isomorphism. We say that an \'{e}tale $(\varphi, \mathcal{O}_{\mathcal{E}})$-module is \textit{projective} (resp. \textit{torsion}) if the underlying $\mathcal{O}_{\mathcal{E}}$-module $\mathcal{M}$ is projective (resp. $p$-power torsion).	
\end{defn}

\noindent Let $\mathrm{Mod}_{\mathcal{O}_{\mathcal{E}}}$ denote the category of \'{e}tale $(\varphi, \mathcal{O}_{\mathcal{E}})$-modules whose morphisms are $\mathcal{O}_\mathcal{E}$-module maps compatible with Frobenius. Let  $\mathrm{Mod}_{\mathcal{O}_{\mathcal{E}}}^{\mathrm{pr}}$ and  $\mathrm{Mod}_{\mathcal{O}_{\mathcal{E}}}^{\mathrm{tor}}$ respectively denote the full subcategories of projective and torsion objects. Note that we have a natural notion of a subquotient, direct sum, and tensor product for \'{e}tale $(\varphi, \mathcal{O}_\mathcal{E})$-modules, and duality is defined for projective and torsion objects.

\begin{lem} \label{lem:2.3}
Let $\mathcal{M} \in \mathrm{Mod}_{\mathcal{O}_{\mathcal{E}}}^{\mathrm{tor}}$ be a torsion \'etale $\varphi$-module annihilated by $p$. Then $\mathcal{M}$ is a projective $\mathcal{O}_{\mathcal{E}}/p\mathcal{O}_{\mathcal{E}}$-module.	
\end{lem}

\begin{proof}
This follows from essentially the same proof as in \cite[Lemma 7.10]{andreatta}.
\end{proof}

We consider $W(\overline{R}^\flat[\frac{1}{\underline{\varpi}}])$ as an $\mathcal{O}_\mathcal{E}$-algebra via mapping $u$ to the Teichm\"uller lift $[\underline{\varpi}]$ of $\underline{\varpi}$, and let $\mathcal{O}_\mathcal{E}^{\mathrm{ur}}$ be the integral closure of $\mathcal{O}_\mathcal{E}$ in $W(\overline{R}^\flat[\frac{1}{\underline{\varpi}}])$. Let $\widehat{\mathcal{O}}_{\mathcal{E}}^{\mathrm{ur}}$ be its $p$-adic completion. Since $\mathcal{O}_\mathcal{E}$ is normal, we have $\mathrm{Aut}_{\mathcal{O}_\mathcal{E}}(\mathcal{O}_\mathcal{E}^{\mathrm{ur}}) \cong \mathcal{G}_{E_{R_\infty}} \coloneqq \pi_1^{\text{\'et}}(\mathrm{Spec}E_{R_\infty})$, and by \cite[Proposition 5.4.54]{gabber-almost} and the almost purity theorem, we have $\mathcal{G}_{E_{R_\infty}} \cong \mathcal{G}_{\tilde{E}_{R_\infty}} \cong \mathcal{G}_{\tilde{R}_\infty}$. This induces $\mathcal{G}_{\tilde{R}_\infty}$-action on $\widehat{\mathcal{O}}_{\mathcal{E}}^{\mathrm{ur}}$. The following is proved in \cite{kim-groupscheme-relative}.

\begin{lem} \label{lem:2.4} \emph{(cf. \cite[Lemma 7.5 and 7.6]{kim-groupscheme-relative})}
We have $(\widehat{\mathcal{O}}_{\mathcal{E}}^{\mathrm{ur}})^{\mathcal{G}_{\tilde{R}_\infty}} = \mathcal{O}_\mathcal{E}$ and the same holds modulo $p^n$. Furthermore, there exists a unique $\mathcal{G}_{\tilde{R}_\infty}$-equivariant ring endomorphism $\varphi$ on $\widehat{\mathcal{O}}_{\mathcal{E}}^{\mathrm{ur}}$ lifting the $p$-th power map on $\widehat{\mathcal{O}}_{\mathcal{E}}^{\mathrm{ur}}/(p)$ and extending $\varphi$ on $\mathcal{O}_\mathcal{E}$. The inclusion $\widehat{\mathcal{O}}_{\mathcal{E}}^{\mathrm{ur}} \hookrightarrow W(\overline{R}^\flat[\frac{1}{\underline{\varpi}}])$ is $\varphi$-equivariant where the latter ring is given the Witt vector Frobenius.
\end{lem}
 
Let $\mathrm{Rep}_{\mathbf{Z}_p}(\mathcal{G}_{\tilde{R}_\infty})$ be the category of $\mathbf{Z}_p$-representations of $\mathcal{G}_{\tilde{R}_\infty}$, and let $\mathrm{Rep}_{\mathbf{Z}_p}^{\mathrm{pr}}(\mathcal{G}_{\tilde{R}_\infty})$ and $\mathrm{Rep}_{\mathbf{Z}_p}^{\mathrm{tor}}(\mathcal{G}_{\tilde{R}_\infty})$ respectively denote the full subcategories of free and torsion objects. For $\mathcal{M} \in \mathrm{Mod}_{\mathcal{O}_\mathcal{E}}$ and $T \in \mathrm{Rep}_{\mathbf{Z}_p}(\mathcal{G}_{\tilde{R}_\infty})$, we define $T(\mathcal{M}) \coloneqq (\mathcal{M}\otimes_{\mathcal{O}_\mathcal{E}}\widehat{\mathcal{O}}_{\mathcal{E}}^{\mathrm{ur}})^{\varphi = 1}$ and $\mathcal{M}(T) \coloneqq (T\otimes_{\mathbf{Z}_p}\widehat{\mathcal{O}}_{\mathcal{E}}^{\mathrm{ur}})^{\mathcal{G}_{\tilde{R}_\infty}}$. For a torsion \'etale $\varphi$-module $\mathcal{M} \in \mathrm{Mod}_{\mathcal{O}_{\mathcal{E}}}^{\mathrm{tor}}$, we define its \textit{length} to be the length of $\mathcal{M}\otimes_{\mathcal{O}_{\mathcal{E}}}(\mathcal{O}_{\mathcal{E}})_{(p)}$ as an $(\mathcal{O}_{\mathcal{E}})_{(p)}$-module. 

\begin{prop} \label{prop:2.5} \emph{(cf. \cite[Proposition 7.7]{kim-groupscheme-relative})} The assignments $T(\cdot)$ and $\mathcal{M}(\cdot)$ are exact equivalences (inverse of each other) of $\otimes$-categories between $\mathrm{Mod}_{\mathcal{O}_{\mathcal{E}}}$ and $\mathrm{Rep}_{\mathbf{Z}_p}(\mathcal{G}_{\tilde{R}_\infty})$. Moreover, $T(\cdot)$ and $\mathcal{M}(\cdot)$ restrict to rank-preserving equivalence of categories between $\mathrm{Mod}_{\mathcal{O}_{\mathcal{E}}}^{\mathrm{pr}}$ and $\mathrm{Rep}_{\mathbf{Z}_p}^{\mathrm{pr}}(\mathcal{G}_{\tilde{R}_\infty})$ and length-preserving equivalence of categories between $\mathrm{Mod}_{\mathcal{O}_{\mathcal{E}}}^{\mathrm{tor}}$ and $\mathrm{Rep}_{\mathbf{Z}_p}^{\mathrm{tor}}({\mathcal{G}}_{\tilde{R}_\infty})$. In both cases, $T(\cdot)$ and $\mathcal{M}(\cdot)$ commute with taking duals. 
\end{prop}

\begin{proof}
This is \cite[Proposition 7.7]{kim-groupscheme-relative}. We remark here for some additional details. Note that $E_{R_\infty}$ is a normal domain and $\pi_1^{\text{\'{e}t}}(\mathrm{Spec} \mathbf{E}_{R_{\infty}}) \cong \mathcal{G}_{\tilde{R}_\infty}$. Given Lemma \ref{lem:2.3}, the assertion therefore follows from the usual d\'evissage and [Katz, Lemma 4.1.1]. Note that both functors $T(\cdot)$ and $\mathcal{M}(\cdot)$ are a priori left exact by definition, and exactness can be proved by the same argument as in the proof of \cite[Theorem 7.11]{andreatta}. 	
\end{proof}

Suppose $S_0$ is another relative base ring over $W(k)\langle t_1^{\pm 1}, \ldots, t_m^{\pm 1}\rangle$ as in Section \ref{sec:2.1} equipped with a choice of Frobenius, and suppose $b: R_0 \hookrightarrow S_0$ be a $\varphi$-equivariant $W(k)\langle t_1^{\pm 1}, \ldots, t_m^{\pm 1}\rangle$-algebra map which is injective. Let $b: R = R_0\otimes_{W(k)}\mathcal{O}_K \hookrightarrow S \coloneqq S_0\otimes_{W(k)}\mathcal{O}_K$ be the induced injective map. By choosing a common geometric point we have an injective map $\overline{R} \hookrightarrow \overline{S}$, and this induces an embedding $\tilde{R}_{\infty} \hookrightarrow \tilde{S}_{\infty}$ by the constructions given in equation (\ref{eq:2.1}). Hence, the corresponding map of Galois groups $\mathcal{G}_S \rightarrow \mathcal{G}_R$ restricts to $\mathcal{G}_{\tilde{S}_\infty} \rightarrow \mathcal{G}_{\tilde{R}_\infty}$. Let $\mathfrak{S}_S = S_0[\![u]\!]$ and let $\mathcal{O}_{\mathcal{E}, S}$ be the $p$-adic completion of $\mathfrak{S}_S[\frac{1}{u}]$. Let $\mathcal{M}_S(\cdot)$ be the functor for the base ring $S$ constructed similarly as above. Let $T \in \mathrm{Rep}_{\mathbf{Z}_p}^{\mathrm{pr}}(\mathcal{G}_{\tilde{R}_\infty})$. Then $T$ is also a $\mathcal{G}_{\tilde{S}_\infty}$-representation via the map $\mathcal{G}_{\tilde{S}_\infty} \rightarrow \mathcal{G}_{\tilde{R}_\infty}$, and we have the natural isomorphism $\mathcal{M}(T)\otimes_{\mathcal{O}_\mathcal{E}}\mathcal{O}_{\mathcal{E}, S} \cong \mathcal{M}_S(T)$ as \'etale $(\varphi, \mathcal{O}_{\mathcal{E}, S})$-modules by the definition of the functors $\mathcal{M}(\cdot)$ and $T(\cdot)$ and by Proposition \ref{prop:2.5}.

\section{Relative Breuil-Kisin classification} \label{sec:3}

We now explain the classification of $p$-divisible groups over $\mathrm{Spec} R$ via Kisin modules, which is proved in \cite{kisin-crystalline} when $R = \mathcal{O}_K$ and generalized in \cite{kim-groupscheme-relative} for the relative case. Denote by $E(u)$ the Eisenstein polynomial for the extension $K$ over $W(k)[\frac{1}{p}]$, and let $\mathfrak{S} = R_0[\![u]\!]$ as above. 

\begin{defn} \label{def:3.1}
Denote by $\mathrm{Kis}^1(\mathfrak{S})$ the category of pairs $(\mathfrak{M}, \varphi_{\mathfrak{M}})$ where
\begin{itemize}
\item $\mathfrak{M}$ is a finitely generated projective $\mathfrak{S}$-module;
\item $\varphi_{\mathfrak{M}}: \mathfrak{M} \rightarrow \mathfrak{M}$ is a $\varphi$-semilinear map such that $\mathrm{coker}(1\otimes\varphi_{\mathfrak{M}})$ is annihilated by $E(u)$.
\end{itemize} 
The morphisms are $\mathfrak{S}$-module maps compatible with Frobenius.	
\end{defn}

\noindent Note that for $(\mathfrak{M}, \varphi_{\mathfrak{M}}) \in \mathrm{Kis}^1(\mathfrak{S})$, $1\otimes\varphi_{\mathfrak{M}}: \varphi^*\mathfrak{M} \rightarrow \mathfrak{M}$ is injective since $\mathfrak{M}$ is finite projective over $\mathfrak{S}$ and $\mathrm{coker}(1\otimes\varphi_{\mathfrak{M}})$ is killed by $E(u)$. Consider the composite $\mathfrak{S} \twoheadrightarrow \mathfrak{S}/u\mathfrak{S} = R_0 \stackrel{\varphi}{\rightarrow} R_0$.

\begin{defn} \label{def:3.2}
A \textit{Kisin module} of height $1$ is a tuple $(\mathfrak{M}, \varphi_{\mathfrak{M}}, \nabla_{\mathfrak{M}})$ such that
\begin{itemize}
\item $(\mathfrak{M}, \varphi_{\mathfrak{M}}) \in \mathrm{Kis}^1(\mathfrak{S})$;
\item Let $\mathcal{N} \coloneqq \mathfrak{M}\otimes_{\mathfrak{S}, \varphi}R_0$ equipped with the Frobenius $\varphi_{\mathfrak{M}}\otimes\varphi_{R_0}$. Then $\nabla_{\mathfrak{M}}: \mathcal{N} \rightarrow \mathcal{N}\otimes_{R_0}\widehat{\Omega}_{R_0}$ is a topologically quasi-nilpotent integrable connection commuting with Frobenius.	
\end{itemize}
Here, $\nabla_{\mathfrak{M}}$ being topologically quasi-nilpotent means that the induced connection on $\mathcal{N}/p\mathcal{N}$ is nilpotent. Denote by $\mathrm{Kis}^1(\mathfrak{S}, \nabla)$ the category of Kisin modules of height $1$ whose morphisms are $\mathfrak{S}$-module maps compatible with Frobenius and connection.
\end{defn}

The following theorem classifying the $p$-divisible groups is proved in \cite{kim-groupscheme-relative}.

\begin{thm} \label{thm:3.3} \emph{(cf. \cite[Corollary 6.7 and Remark 6.9]{kim-groupscheme-relative})}
There exists an exact anti-equivalence of categories 
\[
\mathfrak{M}^*: \{p\mbox{-divisible groups over } \mathrm{Spec}R\} \rightarrow \mathrm{Kis}^1(\mathfrak{S}, \nabla).
\]	
Let $S_0$ be another base ring satisfying the condition as in Section \ref{sec:2.1} and equipped with a Frobenius, and let $b: R_0 \rightarrow S_0$ be a $\varphi$-equivariant map. Then the formation of $\mathfrak{M}^*$ commutes with the base change $R \rightarrow S \coloneqq S_0\otimes_{W(k)}\mathcal{O}_K$ induced $\mathcal{O}_K$-linearly from $b$.   
\end{thm}

Note that if $(\mathfrak{M}, \varphi_{\mathfrak{M}}) \in \mathrm{Kis}^1(\mathfrak{S})$, then $(\mathfrak{M}\otimes_{\mathfrak{S}}\mathcal{O}_{\mathcal{E}}, \varphi_{\mathfrak{M}}\otimes\varphi_{\mathcal{O}_{\mathcal{E}}})$ is a projective \'etale $(\varphi, \mathcal{O}_{\mathcal{E}})$-module since $1\otimes\varphi_{\mathfrak{M}}$ is injective and its cokernel is killed by $E(u)$ which is a unit in $\mathcal{O}_\mathcal{E}$. If $G_R$ is a $p$-divisible group over $R$, its Tate module is given by $T_p(G_R) \coloneqq \mathrm{Hom}_{\overline{R}}(\mathbf{Q}_p/\mathbf{Z}_p, G_R\times_R \overline{R})$, which is a finite free $\mathbf{Z}_p$-representation of $\mathcal{G}_R$. By \cite[Corollary 8.2]{kim-groupscheme-relative}, we have a natural $\mathcal{G}_{\tilde{R}_\infty}$-equivariant isomorphism $T^\vee(\mathfrak{M}^*(G_R)\otimes_{\mathfrak{S}}\mathcal{O}_\mathcal{E}) \cong T_p(G_R)$ where $T^\vee(\mathfrak{M}^*(G_R)\otimes_{\mathfrak{S}}\mathcal{O}_\mathcal{E})$ denotes the dual of $T(\mathfrak{M}^*(G_R)\otimes_{\mathfrak{S}}\mathcal{O}_\mathcal{E})$.

\section{Construction of Kisin modules} \label{sec:4}

Throughout this section, we assume $e< p-1$. We denote $\mathfrak{S}_n \coloneqq \mathfrak{S}/p^n\mathfrak{S}$ for positive integers $n \geq 1$. Let $T$ be a crystalline $\mathcal{G}_R$-representation which is free over $\mathbf{Z}_p$ of rank $d$ with Hodge-Tate weights in $[0, 1]$. Let $\mathcal{M} \coloneqq \mathcal{M}^\vee(T)$ be the associated \'etale $(\varphi, \mathcal{O}_{\mathcal{E}})$-module, where $\mathcal{M}^\vee(T)$ denotes the dual of $\mathcal{M}(T)$. For each integer $n \geq 1$, denote $\mathcal{M}_n = \mathcal{M}/p^n\mathcal{M}$. Note that $\mathcal{M}_n \cong \mathcal{M}^{\vee}(T/p^nT)$. On the other hand, consider the map $b_L: R \rightarrow \mathcal{O}_L$ as in Section \ref{sec:2.1}. $T$ is also a crystalline $\mathcal{G}_{\mathcal{O}_{L}}$-representation with Hodge-Tate weights in $[0, 1]$, so by \cite[Theorem 6.10]{brinon-trihan}, there exists a $p$-divisible group $G_{\mathcal{O}_L}$ over $\mathcal{O}_L$ such that $T_p(G_{\mathcal{O}_L}) \cong T$ as $\mathcal{G}_{\mathcal{O}_L}$-representations. Let $(\mathfrak{M}_{\mathcal{O}_L}, \nabla_{\mathfrak{M}_{\mathcal{O}_L}}) \coloneqq \mathfrak{M}^*(G_{\mathcal{O}_{L}}) \in \mathrm{Kis}^1(\mathfrak{S}_{\mathcal{O}_L}, \nabla)$ be the associated Kisin module over $\mathfrak{S}_{\mathcal{O}_L}$. Denote $\mathfrak{M}_{\mathcal{O}_L, n} = \mathfrak{M}_{\mathcal{O}_L}/p^n\mathfrak{M}_{\mathcal{O}_L}$. The map between the Galois groups $\mathcal{G}_{\mathcal{O}_{L}} \rightarrow \mathcal{G}_R$ restricts to $\mathcal{G}_{\widetilde{\mathcal{O}}_{L}, \infty} \rightarrow \mathcal{G}_{\tilde{R}_\infty}$. Hence, we have the natural isomorphism $\mathcal{M}\otimes_{\mathcal{O}_{\mathcal{E}}}\mathcal{O}_{\mathcal{E}, \mathcal{O}_L} \cong \mathfrak{M}_{\mathcal{O}_L}\otimes_{\mathfrak{S}_{\mathcal{O}_L}}\mathcal{O}_{\mathcal{E}, \mathcal{O}_L}$ of \'etale $(\varphi, \mathcal{O}_{\mathcal{E}, \mathcal{O}_L})$-modules. Let $\mathcal{M}_{\mathcal{O}_L} \coloneqq \mathcal{M}\otimes_{\mathcal{O}_{\mathcal{E}}}\mathcal{O}_{\mathcal{E}, \mathcal{O}_L}$ and $\mathcal{M}_{\mathcal{O}_L, n} \coloneqq \mathcal{M}_{\mathcal{O}_L}/p^n\mathcal{M}_{\mathcal{O}_L}$. 

For each $n \geq 1$, we define
\[
\mathfrak{M}_n \coloneqq \mathcal{M}_n \cap \mathfrak{M}_{\mathcal{O}_L, n}
\]
where the intersection is taken as $\mathfrak{S}$-submodules of $\mathcal{M}_{\mathcal{O}_L, n}$. The Frobenius endomorphisms on $\mathcal{M}_n$ and $\mathfrak{M}_{\mathcal{O}_L, n}$ induce a Frobenius endomorphism on $\mathfrak{M}_n$. Since the Frobenius on $\mathcal{M}_{\mathcal{O}_L, n}$ is injective, we have the injective $\mathfrak{S}$-module morphism
\[
1\otimes\varphi: \mathfrak{S}\otimes_{\varphi, \mathfrak{S}}\mathfrak{M}_n \rightarrow \mathfrak{M}_n
\]
for each $n$.

\begin{lem} \label{lem:4.1}
$\mathfrak{M}_n$ is a finitely generated $\mathfrak{S}_n$-module. Furthermore, we have $\varphi$-equivariant isomorphisms
\[
\mathfrak{M}_n \otimes_{\mathfrak{S}}\mathcal{O}_{\mathcal{E}} \cong \mathcal{M}_n
\]	
and
\[
\mathfrak{M}_n\otimes_{\mathfrak{S}}\mathfrak{S}_{\mathcal{O}_L} \cong \mathfrak{M}_{\mathcal{O}_L, n}.
\]
\end{lem}

\begin{proof}
We first prove that $\mathfrak{M}_n$ is finite over $\mathfrak{S}_n$. Note that $\mathfrak{M}_{\mathcal{O}_L, n}$ is free over $\mathfrak{S}_{\mathcal{O}_L, n}$ of rank $d$, and choose a basis $\{e_1, \ldots, e_d\}$ of $\mathfrak{M}_{\mathcal{O}_L, n}$. On the other hand, since $\mathcal{M}_n$ is projective over $\mathfrak{S}_n[\frac{1}{u}]$ of rank $d$, there exists a non-zero divisor $g \in \mathfrak{S}_n$ such that $\mathcal{M}_n[\frac{1}{g}]$ is free of rank $d$ over $\mathfrak{S}_n[\frac{1}{u}][\frac{1}{g}]$. Since $\mathcal{M}_n$ is finite over $\mathfrak{S}_n[\frac{1}{u}]$, we can choose a basis $\{f_1, \ldots, f_d\}$ of $\mathcal{M}_n[\frac{1}{g}]$ over $\mathfrak{S}_n[\frac{1}{u}][\frac{1}{g}]$ such that letting $\mathfrak{N}$ to be the $\mathfrak{S}_n$-submodule of $\mathcal{M}_n[\frac{1}{g}]$ generated by $f_1, \ldots, f_d$, we have $\mathcal{M}_n \subset \mathfrak{N}[\frac{1}{u}]$ as $\mathfrak{S}_n[\frac{1}{u}]$-modules. It suffices to show that $\mathfrak{M}_n \subset \frac{1}{u^h} \cdot \mathfrak{N}$ as $\mathfrak{S}_n$-modules for some integer $h \geq 1$. We have 
\[
(f_1, \ldots, f_d)^{t} = A\cdot (e_1, \ldots, e_d)^{t}
\]  
where $A$ is an invertible $d \times d$ matrix with entries in $\mathfrak{S}_{\mathcal{O}_L, n}[\frac{1}{u}][\frac{1}{g}]$. Consider the intersection $\mathfrak{N}[\frac{1}{u}] \cap \mathfrak{M}_{\mathcal{O}_L, n}$ as submodules of $\mathfrak{M}_{\mathcal{O}_L, n}[\frac{1}{u}][\frac{1}{g}]$. For an element $x = b_1 f_1+\cdots+b_d f_d \in \mathfrak{N}[\frac{1}{u}]$ with $b_1, \ldots, b_d \in \mathfrak{S}_n[\frac{1}{u}]$, we have $x \in \mathfrak{M}_{\mathcal{O}_L, n}$ if and only if
\[
(b_1, \ldots, b_d)\cdot A = (c_1, \ldots, c_d)
\]          	
for some $c_1, \ldots, c_d \in \mathfrak{S}_{\mathcal{O}_{L}, n}$. Then $(b_1, \ldots, b_d) = A^{-1}(c_1, \ldots, c_d)$, which implies that $\mathfrak{N}[\frac{1}{u}] \cap \mathfrak{M}_{\mathcal{O}_L, n} \subset \frac{1}{u^h} \cdot \mathfrak{N}$ as $\mathfrak{S}_n$-modules for some integer $h \geq 1$. Since $\mathfrak{M}_n \subset \mathfrak{N}[\frac{1}{u}] \cap \mathfrak{M}_{\mathcal{O}_L, n}$, this shows the first statement. 

We have
\[
\mathfrak{M}_n \otimes_{\mathfrak{S}}\mathcal{O}_{\mathcal{E}} \cong \mathfrak{M}_n[\frac{1}{u}] \cong \mathcal{M}_n \cap \mathcal{M}_{\mathcal{O}_L, n} = \mathcal{M}_n
\]
and hence the first isomorphism. On the other hand, since $\mathfrak{S} \rightarrow \mathfrak{S}_{\mathcal{O}_L}$ is flat and $\mathfrak{M}_{\mathcal{O}_L, n}$ is finite free over $\mathfrak{S}_{\mathcal{O}_L, n}$, we have
\[
\begin{split}
\mathfrak{M}_n\otimes_{\mathfrak{S}}\mathfrak{S}_{\mathcal{O}_L} & \cong (\mathcal{M}_n\otimes_{\mathfrak{S}}\mathfrak{S}_{\mathcal{O}_L}) \cap (\mathfrak{M}_{\mathcal{O}_L, n}\otimes_{\mathfrak{S}}\mathfrak{S}_{\mathcal{O}_L}) = \mathcal{M}_{\mathcal{O}_L, n} \cap (\mathfrak{M}_{\mathcal{O}_L, n}\otimes_{\mathfrak{S}}\mathfrak{S}_{\mathcal{O}_L})\\
& \cong (\mathfrak{M}_{\mathcal{O}_L, n}\otimes_{\mathfrak{S}_n}\mathfrak{S}_n[\frac{1}{u}]) \cap (\mathfrak{M}_{\mathcal{O}_L, n}\otimes_{\mathfrak{S}_n}\mathfrak{S}_{\mathcal{O}_L, n}) \cong \mathfrak{M}_{\mathcal{O}_L, n}
\end{split}
\]
by $\mathfrak{S}_n[\frac{1}{u}] \cap \mathfrak{S}_{\mathcal{O}_L, n} = \mathfrak{S}_n$.
\end{proof}

\begin{lem} \label{lem:4.2}
The cokernel of the $\mathfrak{S}$-module map $1\otimes\varphi: \mathfrak{S}\otimes_{\varphi, \mathfrak{S}}\mathfrak{M}_n \rightarrow \mathfrak{M}_n$ is killed by $E(u)$. 	
\end{lem}

\begin{proof}
Let $x \in \mathfrak{M}_n$. There exists a unique $y_1 \in \mathcal{O}_{\mathcal{E}}\otimes_{\varphi, \mathcal{O}_{\mathcal{E}}}\mathcal{M}_n \cong \mathfrak{S}\otimes_{\varphi, \mathfrak{S}} \mathcal{M}_n$ such that $(1\otimes\varphi)(y_1) = E(u)x$. On the other hand, there exists a unique $y_2 \in \mathfrak{S}_{\mathcal{O}_L}\otimes_{\varphi, \mathfrak{S}_{\mathcal{O}_L}}\mathfrak{M}_{\mathcal{O}_L, n}$ such that $(1\otimes\varphi)(y_2) = E(u)x$. Then we have $y_1 = y_2 \in (\mathfrak{S}\otimes_{\varphi, \mathfrak{S}}\mathcal{M}_n) \cap (\mathfrak{S}_{\mathcal{O}_L}\otimes_{\varphi, \mathfrak{S}_{\mathcal{O}_L}}\mathfrak{M}_{\mathcal{O}_L, n})$.

Since $\mathcal{O}_{L_0}/p\mathcal{O}_{L_0}$ has a finite $p$-basis given by $t_1, \ldots, t_m \in R_0/pR_0$ which also gives a $p$-basis of $R_0/pR_0$, the natural map $\mathfrak{S}\otimes_{\varphi, \mathfrak{S}}\mathfrak{M}_{\mathcal{O}_L, n} \rightarrow \mathfrak{S}_{\mathcal{O}_L}\otimes_{\varphi, \mathfrak{S}_{\mathcal{O}_L}}\mathfrak{M}_{\mathcal{O}_L, n}$ is an isomorphism. Hence, 
\[
y_1 \in (\mathfrak{S}\otimes_{\varphi, \mathfrak{S}}\mathcal{M}_n) \cap (\mathfrak{S}\otimes_{\varphi, \mathfrak{S}}\mathfrak{M}_{\mathcal{O}_L, n}) \cong \mathfrak{S}\otimes_{\varphi, \mathfrak{S}} (\mathcal{M}_n \cap \mathfrak{M}_{\mathcal{O}_L, n}) =\mathfrak{S}\otimes_{\varphi, \mathfrak{S}} \mathfrak{M}_n.
\]	
since $\varphi: \mathfrak{S} \rightarrow \mathfrak{S}$ is flat by \cite[Lemma 7.1.8]{brinon-relative}. This proves the assertion.
\end{proof}

\noindent For any finite $\mathfrak{S}$-module $\mathfrak{N}$ equipped with a $\varphi$-semilinear endomorphism $\varphi: \mathfrak{N} \rightarrow \mathfrak{N}$, say $\mathfrak{N}$ has $E(u)$-\textit{height} $\leq 1$ if there exists a $\mathfrak{S}$-module map $\psi: \mathfrak{N} \rightarrow \varphi^*\mathfrak{N} = \mathfrak{S}\otimes_{\varphi, \mathfrak{S}}\mathfrak{N}$ such that the composite
\[
\varphi^*\mathfrak{N} \stackrel{1\otimes\varphi}{\longrightarrow} \mathfrak{N} \stackrel{\psi}{\rightarrow} \varphi^*\mathfrak{N}
\] 
is $E(u)\cdot \mathrm{Id}_{\varphi^*\mathfrak{N}}$. By Lemma \ref{lem:4.2}, $\mathfrak{M}_n$ has $E(u)$-height $\leq 1$.

For each maximal ideal $\mathfrak{q} \in \mathrm{mSpec} R_0$, consider $b_{\mathfrak{q}}: R \rightarrow R_{\mathfrak{q}}$ as in Section \ref{sec:2.1}. By choosing a common geometric point, we have the induced map of Galois groups $\mathcal{G}_{R_{\mathfrak{q}}} \rightarrow \mathcal{G}_R$ which ristricts to $\mathcal{G}_{\widetilde{R}_{\mathfrak{q}, \infty}} \rightarrow \mathcal{G}_{\tilde{R}, \infty}$, and $T$ is a crystalline $\mathcal{G}_{R_{\mathfrak{q}}}$-representation with Hodge-Tate weights in $[0, 1]$. Denote $\mathfrak{S}_{\mathfrak{q}} \coloneqq \widehat{R_{0, \mathfrak{q}}}[\![u]\!]$. 

\begin{prop} \label{prop:4.3}
For each integer $n \geq 1$, $\mathfrak{M}_n$ is projective over $\mathfrak{S}_n$ of rank $d$.	
\end{prop}

\begin{proof}
Let $\mathfrak{q}$ be a maximal ideal of $R_0$, and let $\mathfrak{N}_{n} \coloneqq \mathfrak{M}_n\otimes_{\mathfrak{S}}\mathfrak{S}_{\mathfrak{q}}$ equipped with the induced Frobenius endomorphism. Then we have the induced $\mathfrak{S}_{\mathfrak{q}}$-linear map $\psi: \mathfrak{N}_{n} \rightarrow \mathfrak{S}_{\mathfrak{q}}\otimes_{\varphi, \mathfrak{S}_{\mathfrak{q}}}\mathfrak{N}_{n}$ such that the composite 
\[
\mathfrak{S}_{\mathfrak{q}}\otimes_{\varphi, \mathfrak{S}_{\mathfrak{q}}}\mathfrak{N}_n \stackrel{1\otimes\varphi}{\longrightarrow} \mathfrak{N}_{n} \stackrel{\psi}{\rightarrow} \mathfrak{S}_{\mathfrak{q}}\otimes_{\varphi, \mathfrak{S}_{\mathfrak{q}}}\mathfrak{N}_{n}
\]	
is $E(u)\cdot \mathrm{Id}$. For the isomorphism $\widehat{R_{0, \mathfrak{q}}} \cong \mathcal{O}_{\mathfrak{q}}[\![s_1, \ldots, s_l]\!]$ as above, consider the projection $\mathfrak{S}_{\mathfrak{q}} \rightarrow \mathfrak{S}_{\mathfrak{q}}/(p, s_1, \ldots, s_l) \cong k_{\mathfrak{q}}[\![u]\!]$ where $k_{\mathfrak{q}}\coloneqq \mathcal{O}_{\mathfrak{q}}/(p)$. Denote $\overline{\mathfrak{N}}_n = \mathfrak{N}_n\otimes_{\mathfrak{S}_{\mathfrak{q}}}k_{\mathfrak{q}}[\![u]\!]$ equipped with the induced Frobenius. Then we have the induced $k_{\mathfrak{q}}[\![u]\!]$-linear map $\psi: \overline{\mathfrak{N}}_n \rightarrow k_{\mathfrak{q}}[\![u]\!]\otimes_{\varphi, k_{\mathfrak{q}}[\![u]\!]}\overline{\mathfrak{N}}_n$ such that the composite
\[
k_{\mathfrak{q}}[\![u]\!]\otimes_{\varphi, k_{\mathfrak{q}}[\![u]\!]}\overline{\mathfrak{N}}_n \stackrel{1\otimes\varphi}{\longrightarrow} \overline{\mathfrak{N}}_n \stackrel{\psi}{\rightarrow} k_{\mathfrak{q}}[\![u]\!]\otimes_{\varphi, k_{\mathfrak{q}}[\![u]\!]}\overline{\mathfrak{N}}_n
\] 
is $u^e\cdot \mathrm{Id}$. Since $k_{\mathfrak{q}}[\![u]\!]$ is a principal ideal domain, $\overline{\mathfrak{N}}_n$ is a direct sum of its free part and $u$-torsion part $\overline{\mathfrak{N}}_n \cong \overline{\mathfrak{N}}_{n, \mathrm{free}}\oplus \overline{\mathfrak{N}}_{n, \mathrm{tor}}$ as $k_{\mathfrak{q}}[\![u]\!]$-modules. Furthermore, $\varphi$ maps $\overline{\mathfrak{N}}_{n, \mathrm{tor}}$ into $\overline{\mathfrak{N}}_{n, \mathrm{tor}}$, and hence the above maps induce 
\[
k_{\mathfrak{q}}[\![u]\!]\otimes_{\varphi, k_{\mathfrak{q}}[\![u]\!]}\overline{\mathfrak{N}}_{n, \mathrm{tor}} \stackrel{1\otimes\varphi}{\longrightarrow} \overline{\mathfrak{N}}_{n, \mathrm{tor}} \stackrel{\psi}{\rightarrow} k_{\mathfrak{q}}[\![u]\!]\otimes_{\varphi, k_{\mathfrak{q}}[\![u]\!]}\overline{\mathfrak{N}}_{n, \mathrm{tor}}
\]
whose composite is $u^e\cdot \mathrm{Id}$. 

We claim that $\overline{\mathfrak{N}}_{n, \mathrm{tor}} = 0$. Suppose otherwise. Then $\overline{\mathfrak{N}}_{n, \mathrm{tor}} \cong \bigoplus_{i=1}^b k_{\mathfrak{q}}[\![u]\!]/(u^{a_i})$ for some integers $a_i \geq 1$, and $k_{\mathfrak{q}}[\![u]\!]\otimes_{\varphi, k_{\mathfrak{q}}[\![u]\!]}\overline{\mathfrak{N}}_{n, \mathrm{tor}} \cong \bigoplus_{i=1}^b k_{\mathfrak{q}}[\![u]\!]/(u^{pa_i})$. By taking the appropriate wedge product and letting $a = a_1+\ldots+a_b$, the above maps induce the map of $k_{\mathfrak{q}}[\![u]\!]$-modules
\[
k_{\mathfrak{q}}[\![u]\!]/(u^{pa}) \stackrel{1\otimes\varphi}{\longrightarrow} k_{\mathfrak{q}}[\![u]\!]/(u^{a}) \stackrel{\psi}{\rightarrow} k_{\mathfrak{q}}[\![u]\!]/(u^{pa})
\]
whose composite is equal to $u^{eb}\cdot \mathrm{Id}$. Let $(1\otimes\varphi)(1) = f(u) \in k_{\mathfrak{q}}[\![u]\!]/(u^{a})$, and $\psi(1) = h(u) \in k_{\mathfrak{q}}[\![u]\!]/(u^{pa})$. Then $u^{pa} \mid u^ah(u)$, so $u^{(p-1)a} \mid h(u)$. On the other hand, $f(u)h(u) = u^{eb}$ in $k_{\mathfrak{q}}[\![u]\!]/(u^{pa})$. This implies $u^{(p-1)a} \mid u^{eb}$. But $e < p-1$ and $a \geq b$, so we get a contradiction. Hence, $\overline{\mathfrak{N}}_{n, \mathrm{tor}} = 0$ and $\overline{\mathfrak{N}}_n$ is free over $k_{\mathfrak{q}}[\![u]\!]$ of rank $d$, since by Lemma \ref{lem:4.1} $\overline{\mathfrak{N}}_n[\frac{1}{u}] \cong (\mathcal{M}_n\otimes_{\mathfrak{S}}\mathfrak{S}_{\mathfrak{q}})\otimes_{\mathfrak{S}_{\mathfrak{q}}}k_{\mathfrak{q}}[\![u]\!]$ which is projective over $k_{\mathfrak{q}}(\!(u)\!)$ of rank $d$. Let $b_1, \ldots, b_d \in \mathfrak{N}_n$ be a lift of a basis elements of $\overline{\mathfrak{N}}_n$. By Nakayama's lemma, we have a surjection of $\mathfrak{S}_{\mathfrak{q, n}}$-modules
\[
f: \bigoplus_{i = 1}^d \mathfrak{S}_{\mathfrak{q}, n}\cdot e_i \twoheadrightarrow \mathfrak{N}_n
\] 
given by $e_i \mapsto b_i$. Since $\mathfrak{N}_n[\frac{1}{u}] \cong \mathcal{M}_n\otimes_{\mathfrak{S}}\mathfrak{S}_{\mathfrak{q}}$ is projective over $\mathfrak{S}_{\mathfrak{q}, n}[\frac{1}{u}]$ of rank $d$, $f$ is also injective. Thus, $\mathfrak{N}_n = \mathfrak{M}_n\otimes_{\mathfrak{S}}\mathfrak{S}_{\mathfrak{q}}$ is projective over $\mathfrak{S}_{\mathfrak{q}, n}$ of rank $d$. Since this holds for every $\mathfrak{q} \in \mathrm{mSpec} R_0$, it proves the assertion.  
\end{proof}

\begin{lem} \label{lem:4.4}
Let $\mathfrak{N}$ and $\mathfrak{N}'$ be finite $u$-torsion free $\mathfrak{S}$-modules equipped with Frobenius endomorphisms such that $\mathfrak{N}[\frac{1}{u}]$ and $\mathfrak{N}'[\frac{1}{u}]$ are torsion \'etale $\varphi$-modules. Suppose that $\mathfrak{N}$ and $\mathfrak{N}'$ have $E(u)$-height $\leq 1$ and $\mathfrak{N}[\frac{1}{u}] = \mathfrak{N}'[\frac{1}{u}]$ as \'etale $\varphi$-modules. Then $\mathfrak{N} = \mathfrak{N}'$. 	
\end{lem}

\begin{proof}
Consider $\mathfrak{N}$ and $\mathfrak{N}'$ as $\mathfrak{S}$-submodules of $\mathfrak{N}[\frac{1}{u}]$. Let $\mathfrak{L}$ be the cokernel of the embedding $\mathfrak{N} \hookrightarrow \mathfrak{N}+\mathfrak{N}'$ of $\mathfrak{S}$-modules. Note that $\mathfrak{S}\otimes_{\varphi, \mathfrak{S}}(\mathfrak{N}+\mathfrak{N}') \cong \mathfrak{S}\otimes_{\varphi, \mathfrak{S}}\mathfrak{N}+\mathfrak{S}\otimes_{\varphi, \mathfrak{S}}\mathfrak{N}'$ since $\varphi: \mathfrak{S} \rightarrow \mathfrak{S}$ is flat. Thus, $\mathfrak{N}+\mathfrak{N}'$ has $E(u)$-height $\leq 1$, and $\mathfrak{L}$ has $E(u)$-height $\leq 1$. Since  $\mathfrak{L}[\frac{1}{u}] = 0$, we deduce similarly as in the proof of Proposition \ref{prop:4.3} that $\mathfrak{L} = 0$. So $\mathfrak{N} = \mathfrak{N}+\mathfrak{N}'$. Similarly, $\mathfrak{N}' = \mathfrak{N}+\mathfrak{N}'$.
\end{proof}

It is clear that both $p\mathfrak{M}_{n +1}$ and $\mathfrak{M}_n$ are $u$-torsion free,  have $E(u)$-height $\leq 1$ and $ p\mathfrak{M}_{n+1}[\frac{1}{u}]= p \mathcal M_{n+1} \cong  \mathcal M_n= \mathfrak{M}_n[\frac 1 u] $. We conclude the following: 
\begin{prop} \label{prop:4.5}
For each $n \geq 1$, we have a $\varphi$-equivariant isomorphism
\[
p\mathfrak{M}_{n+1} \cong \mathfrak{M}_n.
\]	
\end{prop}



\noindent By Lemma \ref{lem:4.2}, Proposition \ref{prop:4.3} and \ref{prop:4.5}, if we define the $\mathfrak{S}$-module
\[
\mathfrak{M} \coloneqq \varprojlim_n \mathfrak{M_n},
\]	
then $\mathfrak{M} \in \mathrm{Kis}^1(\mathfrak{S})$. Note that we have a $\varphi$-equivariant isomorphism $\mathfrak{M}\otimes_{\mathfrak{S}}\mathfrak{S}_{\mathcal{O}_L} \cong \mathfrak{M}_{\mathcal{O}_L}$ by Lemma \ref{lem:4.1}.

\begin{rem} \label{rem:4.6}
Analogous statements hold when $T$ is a crystalline $\mathcal{G}_R$-representation with Hodge-Tate weights in $[0, r]$ for the case $er < p-1$, since \cite{brinon-trihan} constructs more generally a functor from crystalline representations with Hodge-Tate weights in $[0, r]$ to Kisin modules of height $r$ when the base is a complete discrete valuation field whose residue field has a finite $p$-basis.
\end{rem}

To study connections for $\mathfrak{M}$, we first consider the following general situation. Let $A_0$ be a $k$-algebra which is an integral domain. Consider $n$-variables $x_1, \ldots, x_n$, and denote $\underline{x} = (x_1, \ldots, x_n)^t$ and $\underline{x}^{[p]} \coloneqq (x_1^p, \ldots, x_n^p)^t$. An Artin-Schreier system of equations in $n$ variables over $A_0$ is given by
\begin{equation} \label{eq:4.1}
\underline{x} = B\underline{x}^{[p]}+C	
\end{equation}
where $B = (b_{ij})_{1 \leq i, j \leq n} \in M_{n\times n}(A_0)$ is an $n\times n$ matrix with entries in $A_0$ and $C = (c_i)_{1 \leq i \leq n} \in M_{n\times 1}(A_0)$. Let
\[
A_1 \coloneqq A_0[x_1, \ldots, x_n]/(x_1-c_1-\sum_{i = 1}^n b_{1i}x_i^p, \ldots, x_n-c_n-\sum_{i = 1}^n b_{ni}x_i^p),
\]
which is the $A_0$-algebra parametrizing the solutions of equation (\ref{eq:4.1}). $A_0 \rightarrow A_1$ is \'etale by \cite[Theorem 2.4.1(a)]{vasiu}.

\begin{lem} \label{lem:4.7}
There exists a non-zero element $f \in A_0$ which depends only on $B$ such that $A_1[\frac{1}{f}]$ is finite \'etale over $A_0[\frac{1}{f}]$.	
\end{lem}

\begin{proof}
We induct on $n$. Suppose $n = 1$. If $\det{B} \neq 0$, then equation (\ref{eq:4.1}) is equivalent to 
\[ 
x_1^p = B^{-1}x_1-B^{-1}C,
\]
so the assertion holds with $f = \det{B}$. If $\det{B} = 0$, then $B = 0$ and $A_1 \cong A_0$, so the assertion holds trivially. 

For $n \geq 2$, if $\det{B} \neq 0$, then equation (\ref{eq:4.1}) is equivalent to 
\[
\underline{x}^{[p]} = B^{-1}\underline{x}-B^{-1}C.
\]	
Hence, with $f = \det{B}$, $A_1[\frac{1}{f}]$ is finite \'etale over $A_0[\frac{1}{f}]$. Suppose $\det{B} = 0$. Denote by $B^{(i)}$ the $i$-th row of $B$. Then up to renumbering the index for $x_i$'s, we have
\[
\sum_{i=1}^n e_iB^{(i)} = 0
\]
for some non-zero $f_1 \in A_0$ and some $e_i \in A_0[\frac{1}{f_1}]$ with $e_n = 1$. From equation (\ref{eq:4.1}), we get
\[
x_n = -\sum_{i=1}^{n-1} e_ix_i+c_n+\sum_{i=1}^{n-1}c_ie_i.
\] 
Hence, denoting $\underline{x}' = (x_1, \ldots, x_{n-1})^t$, equation (\ref{eq:4.1}) is equivalent to an Artin-Schreier system of equations in $n-1$ variables over $A_0[\frac{1}{f_1}]$
\[
\underline{x}' = B'\underline{x}'^{[p]}+C'
\]
where $B' \in M_{(n-1)\times(n-1)}(A_0[\frac{1}{f_1}])$ and $C' \in M_{(n-1)\times 1}(A_0[\frac{1}{f_1}])$. Note that $B'$ depends only on $B$ and not on $C$. Hence, the assertion follows by induction.
\end{proof}

Let $\mathcal{N} \coloneqq \mathfrak{M}\otimes_{\mathfrak{S}, \varphi}R_0$ equipped with the Frobenius $\varphi_{\mathfrak{M}}\otimes\varphi_{R_0}$. From \cite[Eq. (6.1), (6.2) and Remark 3.13]{kim-groupscheme-relative}, we have the $R_0$-submodule $\mathrm{Fil}^1 \mathcal{N} \subset \mathcal{N}$ associated with $\mathfrak{M} \in \mathrm{Kis}^1(\mathfrak{S})$ such that $p\mathcal{N} \subset \mathrm{Fil}^1 \mathcal{N}$, $\mathcal{N}/\mathrm{Fil}^1 \mathcal{N}$ is projective over $R_0/(p)$, and $(1\otimes\varphi)(\varphi^*\mathrm{Fil}^1 \mathcal{N}) = p\mathcal{N}$ as $R_0$-modules (cf. \cite[Definition 3.4 and 3.6]{kim-groupscheme-relative} for the frame $(R_0, pR_0, R_0/(p), \varphi_{R_0}, \frac{\varphi_{R_0}}{p})$). Fix an $R_0$-direct factor $\mathcal{N}^1 \subset \mathcal{N}$ which lifts $\mathrm{Fil}^1 \mathcal{N}/p\mathcal{N} \subset \mathcal{N}/p\mathcal{N}$, and let $\tilde{\mathcal{N}} \coloneqq R_0\otimes_{\varphi, R_0}(\mathcal{N}+\frac{1}{p}\mathcal{N}^1) \subset R_0[\frac{1}{p}]\otimes_{\varphi, R_0}\mathcal{N}$.

Let $\mathrm{Spf}(A, p) \rightarrow \mathrm{Spf}(R_0, p)$ be an \'etale morphism. Note that $A$ is equipped with a unique Frobenius lifting that on $R_0$, and $\widehat{\Omega}_{A} \cong A \hat{\otimes}_{R_0}\widehat{\Omega}_{R_0} \cong \bigoplus_{i=1}^m A\cdot dt_i$. For a connection 
\[
\nabla_{A, n}: A/(p^n) \otimes_{R_0} \mathcal{N} \rightarrow (A/(p^n) \otimes_{R_0} \mathcal{N})\otimes_{A}\widehat{\Omega}_{A}
\] 
on $A/(p^n) \otimes_{R_0} \mathcal{N}$, we say that the Frobenius is \textit{horizontal} if the following diagram commutes:
\[
\begin{CD}
A/(p^n)\otimes_{A}\tilde{\mathcal{N}} @>\varphi^*(\nabla_{A, n})>> A/(p^n)\otimes_{A}\tilde{\mathcal{N}}\otimes_{A}\widehat{\Omega}_{A}\\
@V1\otimes\varphi VV   @VV(1\otimes\varphi)\otimes \mathrm{id}_{\widehat{\Omega}_{A}}V\\
A/(p^n)\otimes_{A}\mathcal{N} @>\nabla_{A, n}>> A/(p^n)\otimes_{A}\mathcal{N}\otimes_{A}\widehat{\Omega}_{A}	
\end{CD}
\]
Here, $\varphi^*(\nabla_{A, n})$ is given by choosing an arbitrary lift of $\nabla_{A, n}$ on $A/(p^{n+1})\otimes_{A}\mathcal{N}$, and $\varphi^*(\nabla_{A, n})$ does not depend on the choice of such a lift (cf. \cite[Section 3.1.1 Equation (9)]{vasiu}).

\begin{prop} \label{prop:4.8}
There exists $\tilde{f} \in R_0$ with $\tilde{f} \notin pR_0$ such that the following holds:

\noindent Let $S_0$ be the $p$-adic completion of $R_0[\frac{1}{\tilde{f}}]$	 equipped with the induced Frobenius, and let $\mathfrak{S}_{S} = S_0[\![u]\!]$. Let $\mathfrak{M}_{S} = \mathfrak{M}\otimes_{\mathfrak{S}}\mathfrak{S}_{S}$ equipped with the induced Frobenius, so $\mathfrak{M}_{S} \in \mathrm{Kis}^1(\mathfrak{S}_{S})$. Then there exists a topologically quasi-nilpotent integrable connection 
\[
\nabla_{\mathfrak{M}_{S}}: (S_0\otimes_{\varphi, \mathfrak{S}_{S}}\mathfrak{M}_{S}) \rightarrow (S_0\otimes_{\varphi, \mathfrak{S}_{S}}\mathfrak{M}_{S})\otimes_{S_0} \widehat{\Omega}_{S_0}  
\] 
such that $\varphi$ is horizontal, and thus $(\mathfrak{M}_{S}, \nabla_{\mathfrak{M}_{S}}) \in \mathrm{Kis}^1(\mathfrak{S}_{S}, \nabla)$. Furthermore, we can choose $\nabla_{\mathfrak{M}_{S}}$ so that $\mathfrak{M}_{S}\otimes_{\mathfrak{S}_{S}}\mathfrak{S}_{\mathcal{O}_L}$ equipped with the induced Frobenius and connection is isomorphic to $(\mathfrak{M}_{\mathcal{O}_L}, \nabla_{\mathfrak{M}_{\mathcal{O}_L}})$ as Kisin modules over $\mathfrak{S}_{\mathcal{O}_L}$.  
\end{prop}

\begin{proof}
Without loss of generality, we may pass to a Zariski open set of $\mathrm{Spf}(R_0, p)$ if necessary so that $\mathcal{N}^1$ and $\mathcal{N}/\mathcal{N}^1$ are free over $R_0$. Fix an $R_0$-basis of $\mathcal{N}$ adapted to the direct factor $\mathcal{N}^1$. Let $\mathrm{Spf}(A, p) \rightarrow \mathrm{Spf}(R_0, p)$ be an \'etale morphism. Consider a connection
\[
\nabla_{A, 1}: A/(p) \otimes_{R_0} \mathcal{N} \rightarrow (A/(p) \otimes_{R_0} \mathcal{N})\otimes_{A}\widehat{\Omega}_{A}
\]  
such that the Frobenius is horizontal. By \cite[Section 3.2 Basic Theorem]{vasiu} and its proof, the set of such connections $\nabla_{A, 1}$ corresponds to the set of solutions over $A/(p)$ of an Artin-Schreier system of equations
\[
\underline{x} = B\underline{x}^{[p]}+C_1
\]
for $\underline{x} = (x_1, \ldots, x_{dm})^t$, where $B \in M_{dm\times dm}(R_0/(p))$ and $C_1 \in M_{dm\times 1}(R_0/(p))$. When $A = \mathcal{O}_{L_0}$, it has a solution given by $\nabla_{\mathfrak{M}_{L_0}}$. Since $\mathcal{O}_{L_0}/(p) \cong \mathrm{Frac}(R_0/(p))$ and $R_0/(p)$ is a unique factorization domain, the solution lies in $(R_0/(p))[\frac{1}{f}]$ for some non-zero $f \in R_0/(p)$ depending only on $B$ by Lemma \ref{lem:4.7} and its proof. Let $\tilde{f} \in R_0$ be a lift of $f$, and let $S_0$ be the $p$-adic completion of $R_0[\frac{1}{\tilde{f}}]$.

For $n \geq 1$, suppose we are given a connection 
\[
\nabla_{S_0, n}: S_0/(p^n)\otimes_{R_0} \mathcal{N} \rightarrow (S_0/(p^n)\otimes_{R_0} \mathcal{N})\otimes_{S_0} \widehat{\Omega}_{S_0}
\] 
such that the Frobenius is horizontal and inducing $\nabla_{\mathfrak{M}_{L_0}} (\mathrm{mod} ~p^n)$ via the natural map $S_0 \rightarrow \mathcal{O}_{L_0}$. By \cite[Section 3.2 Basic Theorem]{vasiu} and its proof, for the choice of a basis of $\mathcal{N}$ as above, the set of connections
\[
\nabla_{S_0, n+1}: S_0/(p^{n+1})\otimes_{R_0} \mathcal{N} \rightarrow (S_0/(p^{n+1})\otimes_{R_0} \mathcal{N})\otimes_{S_0} \widehat{\Omega}_{S_0}
\]
such that the Frobenius is horizontal and lifting $\nabla_{S_0, n}$ corresponds to the set of solutions over $S_0/(p)$ of an Artin-Schreier system of equations
\[
\underline{x} = B\underline{x}^{[p]}+C_{n+1},
\]
where $B$ is the same matrix as above and $C_{n+1} \in M_{dm\times 1}(S_0/(p))$. The solution over $\mathcal{O}_{L_0}/(p)$ given by $\nabla_{\mathfrak{M}_{L_0}}$ lies in $S_0/(p)$ by Lemma \ref{lem:4.7} and its proof. This proves the assertion.
\end{proof}

\begin{prop} \label{prop:4.9}
Let $S_0$ be a ring as given in Proposition \ref{prop:4.8}, and let $S = S_0\otimes_{W(k)}\mathcal{O}_K$. Then there exists a $p$-divisible group $G_S$ over $S$ such that $T_p(G_S) \cong T$ as $\mathcal{G}_S$-representations.	
\end{prop}

\begin{proof}
Let $G_S$ be the $p$-divisible group over $S$ given by $(\mathfrak{M}_{S}, \nabla_{\mathfrak{M}_{S}})$ in Proposition \ref{prop:4.8}. Since $\mathfrak{M}_{S}\otimes_{\mathfrak{S}_{S}}\mathfrak{S}_{\mathcal{O}_L} \cong \mathfrak{M}_{\mathcal{O}_L}$ as Kisin modules, we have $T_p(G_S) \cong T$ as $\mathcal{G}_{\mathcal{O}_L}$-representations. On the other hand, $\mathfrak{M}_{S}\otimes_{\mathfrak{S}_{S}}\mathcal{O}_{\mathcal{E}, S} \cong \mathcal{M}\otimes_{\mathcal{O}_{\mathcal{E}}}\mathcal{O}_{\mathcal{E}, S}$ as \'etale $\varphi$-modules. Hence, $T_p(G_S) \cong T$ as $\mathcal{G}_{\tilde{S}, \infty}$-representations. Since $\mathcal{G}_{\tilde{S}, \infty}$ and $\mathcal{G}_{\mathcal{O}_L}$ generate the Galois group $\mathcal{G}_S$ by Lemma \ref{lem:2.1}, we have $T_p(G_S) \cong T$ as $\mathcal{G}_S$-representations.
\end{proof}

\section{Proof of the main theorem}

In this section, we finish the proof of Theorem \ref{thm:1.2}. We begin by recalling the following well-known lemma about $p$-divisible groups.

\begin{lem} \label{lem:5.1}
Let $R_1$ be an integral domain over $W(k)$ such that $\mathrm{Frac}(R_1)$ has characteristic $0$. Then via the Tate module functor $T_p(\cdot)$, the category of $p$-divisible groups over $R_1[\frac{1}{p}]$ is equivalent to the category of finite free $\mathbf{Z}_p$-representations of $\mathcal{G}_{R_1} = \pi_1^{\text{\'et}}(\mathrm{Spec}R_1[\frac{1}{p}])$. Furthermore, such an equivalence is functorial in the following sense:

Let $R_1 \rightarrow R_2$ be a map of integral domains over $W(k)$ such that $\mathrm{Frac}(R_1)$ and $\mathrm{Frac}(R_2)$ have characteristic $0$. Let $G_{R_1}$ be a $p$-divisible group over $R_1$. Then $T_p(G_{R_1}) \cong T_p(G_{R_1}\times_{R_1}R_2)$ as $\mathcal{G}_{R_2}$-representations.  
\end{lem}

We first consider the case when $R$ is a formal power series ring of dimension $2$. Let $T$ be a crystalline $\mathcal{G}_R$-representation which is finite free over $\mathbf{Z}_p$ and has Hodge-Tate weights in $[0, 1]$.

\begin{prop} \label{prop:5.2}
Suppose $R_0 = \mathcal{O}[\![s]\!]$ for a Cohen ring $\mathcal{O}$, and $e \leq p-1$. Then there exists a $p$-divisible group $G_R$ over $R$ such that $T_p(G_R) \cong T$ as $\mathcal{G}_R$-representations.
\end{prop}

\begin{proof}
Let $G$ be a $p$-divisible group over $R[\frac{1}{p}]$ given by Lemma \ref{lem:5.1} such that $T_p(G) \cong T$ as $\mathcal{G}_R$-representations. It suffices to show that $G$ extends to a $p$-divisible group $G_R$ over $R$.

By \cite[Theorem 6.10]{brinon-trihan}, there exists a $p$-divisible group $G_{\mathcal{O}_L}$ over $\mathcal{O}_L$ extending $G \times_{R[\frac{1}{p}]} L$. For each integer $n \geq 1$, let $A_n$ be the Hopf algebra over $R[\frac{1}{s}][\frac{1}{p}]$ for the finite flat group scheme $(G\times_{R[\frac{1}{p}]}R[\frac{1}{s}][\frac{1}{p}])[p^n]$, and let $B_n$ be the Hopf algebra over $\mathcal{O}_L$ for the finite flat group scheme $G_{\mathcal{O}_L}[p^n]$. Identify $A_n\otimes_{R[\frac{1}{s}][\frac{1}{p}]}L = B_n\otimes_{\mathcal{O}_L}L$ as Hopf algebras over $L$. Note that the $p$-adic completion of $R[\frac{1}{s}]$ is isomorphic to $\mathcal{O}_L$. By \cite[Main Theorem]{beuville-Laszlo} and its proof, the $R[\frac{1}{s}]$-subalgebra $C_n \coloneqq A_n \cap B_n \subset  B_n\otimes_{\mathcal{O}_L}L$ is finite flat over $R[\frac{1}{s}]$. Moreover, $C_n$ is equipped with the Hopf algebra structure induced from $(A_n, B_n)$ such that $C_n\otimes_{R[\frac{1}{s}]}R[\frac{1}{s}][\frac{1}{p}] \cong A_n$ and $C_n\otimes_{R[\frac{1}{s}]}\mathcal{O}_L \cong B_n$. Hence, the datum of finite flat group schemes $((G\times_{R[\frac{1}{p}]}R[\frac{1}{s}][\frac{1}{p}])[p^n], G_{\mathcal{O}_L}[p^n])$ descends to a finite flat group scheme over $R[\frac{1}{s}]$ (up to a unique isomorphism by \cite[Main Theorem]{beuville-Laszlo}).

Thus, we obtain a system of finite flat group schemes $(G_{U, n})_{n \geq 1}$ over $U \coloneqq \mathrm{Spec}R ~\backslash~ \mathrm{pt}$ extending $(G[p^n])_{n \geq 1}$. Here, $\mathrm{pt}$ denotes the closed point given by the maximal ideal of $R$. The natural induced sequence of finite flat group schemes
\[
0 \rightarrow G_{U, 1} \rightarrow G_{U, n+1} \stackrel{\times p}{\longrightarrow} G_{U, n} \rightarrow 0
\]
is exact by fpqc descent. So $(G_{U, n})_{n \geq 1}$ is a $p$-divisible group over $U$ extending $G$. Since $e \leq p-1$, $G_U$ extends to a $p$-divisible group $G_R$ over $R$ by \cite[Theorem 3]{vasiu-zink-purity}.
\end{proof}

Now, let $R_0$ be a general ring satisfying the assumptions in Section \ref{sec:2.1}, and let $R = R_0\otimes_{W(k)} \mathcal{O}_K$ with $e < p-1$. Let $T$ be a crystalline $\mathcal{G}_R$-representation free over $\mathbf{Z}_p$ with Hodge-Tate weights in $[0, 1]$. Denote by $\mathfrak{M}_{\mathfrak{S}}(T)$ the $\mathfrak{S}$-module in  $\mathrm{Kis}^1(\mathfrak{S})$ constructed from $T$ as in Section \ref{sec:4}. Let $\tilde{f} \in R_0$ be an element as in Proposition \ref{prop:4.8}, and let $S_0$ be the $p$-adic completion of $R_0[\frac{1}{\tilde{f}}]$ as in Proposition \ref{prop:4.9}. Let $f \in R_0/pR_0$ be the image of $\tilde{f}$ in the projection $R_0 \rightarrow R_0/(p)$. We only need to consider the case when $f$ is not a unit in $R_0/(p)$. Since $R_0/(p)$ is a UFD,  there exist prime elements $\bar{s}_1, \ldots, \bar{s}_l$ of $R_0/(p)$ dividing $f$. Let $s_1, \ldots, s_l \in R_0$ be any preimages of $\bar{s}_1, \ldots, \bar{s}_l$ respectively. 

For each $i = 1, \ldots, l$, consider the prime ideal $\mathfrak{p}_i = (p, s_i) \subset R_0$ and let $R_0^{(i)} \coloneqq \widehat{R_{0, \mathfrak{p}_i}}$ be the $\mathfrak{p}_i$-adic completion of $R_{0, \mathfrak{p}_i}$. Note that $R_0^{(i)}$ is a formal power series ring over a Cohen ring with Krull dimension $2$. We consider the natural $\varphi$-equivariant map $b_i: R_0 \rightarrow R_0^{(i)}$, which induces $b_i: R \rightarrow R^{(i)}\coloneqq R_0^{(i)}\otimes_{W(k)}\mathcal{O}_K$. On the other hand, let $k_c$ be a field extension of $\mathrm{Frac}(R_0/pR_0)$ which is a composite of the fields $\mathrm{Frac} (R_0^{(i)}/(p))$ for $i = 1, \ldots, l$, and let $k_c^{\mathrm{perf}} = \varinjlim_{\varphi} k_c$ be its direct perfection. By the universal property of $p$-adic Witt vectors, there exists a unique $\varphi$-equivariant map $b_c: R_0 \rightarrow W(k_c^{\mathrm{perf}})$. Moreover, for each $i = 1, \ldots, l$, we have a unique $\varphi$-equivariant embedding $R_0^{(i)} \rightarrow W(k_c^{\mathrm{perf}})$ whose composite with $b_i$ is equal to $b_c$. Note that the natural embedding $R_0 \rightarrow S_0 \cap \bigcap_{i=1}^l R_0^{(i)}$ as subrings of $W(k_c^{\mathrm{perf}})$ is bijective, since $S_0/(p) \cap \bigcap_{i=1}^l (R_0^{(i)}/(p)) = R_0/(p)$ inside $k_c^{\mathrm{perf}}$.

By Proposition \ref{prop:5.2}, there exists a $p$-divisible group $G_i$ over $R^{(i)}$ such that $T_p(G_i) \cong T$ as $\mathcal{G}_{R^{(i)}}$-representations. We have
\[(\mathfrak{M}_{\mathfrak{S}}(T)\otimes_{\mathfrak{S}}\mathcal{O}_{\mathcal{E}})\otimes_{\mathcal{O}_{\mathcal{E}}}\mathcal{O}_{\mathcal{E}, R^{(i)}} \cong \mathfrak{M}^*(G_i)\otimes_{\mathfrak{S}_{R^{(i)}}}\mathcal{O}_{\mathcal{E}, R^{(i)}}
\]
as \'etale $(\varphi, \mathcal{O}_{\mathcal{E}, R^{(i)}})$-modules. Applying Lemma \ref{lem:4.4}, we can deduce that $\mathfrak{M}_{\mathfrak{S}}(T)\otimes_{\mathfrak{S}}\mathfrak{S}_{R^{(i)}} \cong \mathfrak{M}^*(G_i)$ compatibly with Frobenius. 

Let $D = D_{\mathrm{cris}}(T[\frac{1}{p}])$, and denote $\mathfrak{M} = \mathfrak{M}_{\mathfrak{S}}(T)$ and $\mathcal{N} = \mathfrak{M}\otimes_{\mathfrak{S}, \varphi}R_0$. Let $\nabla: D \rightarrow D\otimes_{R_0} \widehat{\Omega}_{R_0}$ be the connection given by the functor $D_{\mathrm{cris}}(\cdot)$. 

\begin{prop} \label{prop:5.3}
There exists a natural $\varphi$-equivariant embedding 
\[
h: \mathcal{N} \hookrightarrow D
\]	
of $R_0$-modules. Furthermore, if we consider $\mathcal{N}$ as an $R_0$-submodule of $D$ via $h$, then $\nabla$ maps $\mathcal{N}$ into $\mathcal{N}\otimes_{R_0} \widehat{\Omega}_{R_0}$. Hence, $\mathfrak{M}$ is a Kisin module of height $1$.
\end{prop}

\begin{proof}
By \cite[Corollary 5.3 and 6.7]{kim-groupscheme-relative}, there exists a natural $\varphi$-equivariant embedding 
\[
h_i: \mathcal{N} \rightarrow D\otimes_{R_0, b_i} R_0^{(i)}
\]	
for each $i = 1, \ldots, l$ such that the connections given by $\mathfrak{M}^*(G_i)$ and $D$ are compatible, and there exists a natural $\varphi$-equivariant embedding $h_c: \mathcal{N} \rightarrow D\otimes_{R_0, b_c} W(k_c^{\mathrm{perf}})$. Moreover, by Proposition \ref{prop:4.9}, there exists a natural $\varphi$-equivariant embedding $h_S: \mathcal{N} \rightarrow D\otimes_{R_0} S_0$ such that the connections given by $\mathfrak{M}^*(G_S)$ and $D$ are compatible. Since the construction of those natural maps  is compatible with $\varphi$-equivariant base changes (cf. \cite[Section 5.5]{kim-groupscheme-relative}), we deduce that the maps $h_1, \ldots, h_l$ and $h_S$ are compatible with one another, in the sense that their composites with the embedding into $D\otimes_{R_0, b_c}W(k_c^{\mathrm{perf}})$ are all equal to $h_c$. Hence, we obtain a $\varphi$-equivariant embedding
\[
h: \mathcal{N} \hookrightarrow (D\otimes_{R_0[\frac{1}{p}]} S_0[\frac{1}{p}]) \cap (\bigcap_{i=1}^l D\otimes_{R_0[\frac{1}{p}], b_i} R_0^{(i)}[\frac{1}{p}]) \cong D\otimes_{R_0[\frac{1}{p}]} (S_0[\frac{1}{p}] \cap \bigcap_{i=1}^l R_0^{(i)}[\frac{1}{p}]) = D,
\]
since $D$ is flat over $R_0[\frac{1}{p}]$ and $S_0[\frac{1}{p}] \cap \bigcap_{i=1}^l R_0^{(i)}[\frac{1}{p}] = R_0[\frac{1}{p}]$. 

Now, identify $\widehat{\Omega}_{R_0} = \bigoplus_{j=1}^m R_0 \cdot dt_j$. Then $\nabla$ maps $\mathcal{N}$ to $\mathcal{N}\otimes_{R_0} (\bigoplus_{j=1}^m R_0[\frac{1}{p}] \cdot dt_j)$. On the other hand, by Proposition \ref{prop:4.8}, Proposition \ref{prop:5.2}, and the compatibility of $D_{\mathrm{cris}}(\cdot)$ with respect to $\varphi$-compatible base changes, we have that $\nabla$ maps $\mathcal{N}$ into $\mathcal{N}\otimes_{R_0} (\bigoplus_{j=1}^m S_0 \cdot dt_j)$ and also into $\mathcal{N}\otimes_{R_0} (\bigoplus_{j=1}^m R_0^{(i)} \cdot dt_j)$ for each $i = 1, \ldots, l$. Since $\mathcal{N}$ is flat over $R_0$ and $S_0 \cap \bigcap_{i=1}^l R_0^{(i)} = R_0$, $\nabla$ maps $\mathcal{N}$ into $\mathcal{N}\otimes_{R_0} (\bigoplus_{j=1}^m R_0 \cdot dt_j)$.
\end{proof}

\begin{thm} \label{thm:5.4}
There exists a $p$-divisible group $G_R$ over $R$ such that $T_p(G_R) \cong T$ as $\mathcal{G}_R$-representations.	
\end{thm}

\begin{proof}
By Proposition \ref{prop:5.3}, we have $\mathfrak{M} \in \mathrm{Kis}^1(\varphi, \nabla)$. Furthermore, $\mathfrak{M}\otimes_{\mathfrak{S}}\mathfrak{S}_{\mathcal{O}_L} \cong \mathfrak{M}_{\mathcal{O}_L}$ as Kisin modules over $\mathfrak{S}_{\mathcal{O}_L}$, since the Frobenius and connection structure on $\mathfrak{M}$ agree with those on $D$. Thus, if $G_R$ is the $p$-divisible group corresponding to $\mathfrak{M}$, then $T_p(G_R) \cong T$ as $\mathcal{G}_{\mathcal{O}_L}$-representations as well as $\mathcal{G}_{\tilde{R}_\infty}$-representations. The assertion follows from Lemma \ref{lem:2.1}.	
\end{proof}

\section{Barsotti-Tate deformation ring} \label{sec:6}

As an application of Theorem \ref{thm:5.4}, we study the geometry of the locus of crystalline representations with Hodge-Tate weights in $[0, 1]$ by using the results in \cite{moon-relativeRaynaud}. Denote by $\mathcal{C}$ the category of topological local $\mathbf{Z}_p$-algebras $A$ satisfying the following conditions:
\begin{itemize}
\item the natural map $\mathbf{Z}_p \rightarrow A/\mathfrak{m}_A$ is surjective, where $\mathfrak{m}_A$ denotes the maximal ideal of $A$;
\item the map from $A$ to the projective limit of its discrete artinian quotients is a topological isomorphism.	
\end{itemize}

\noindent By the first condition, the residue field of $A$ is $\mathbf{F}_p$. The second condition is equivalent to that $A$ is complete and its topology is given by a collection of open ideals $\mathfrak{a} \subset A$ for which $A/\mathfrak{a}$ is aritinian. Morphisms in $\mathcal{C}$ are continuous $\mathbf{Z}_p$-algebra morphisms. 

For $A \in \mathcal{C}$, we mean by an $A$-\textit{representation of} $\mathcal{G}_R$ a finite free $A$-module equipped with a continuous $A$-linear $\mathcal{G}_R$-action. Fix an $\mathbf{F}_p$-representation $V_0$ of $\mathcal{G}_R$ which is absolutely irreducible. For $A \in \mathcal{C}$, a \textit{deformation} of $V_0$ in $A$ is defined to be an isomorphism class of $A$-representations of $V$ of $\mathcal{G}_R$ satisfying $V\otimes_A \mathbf{F}_p \cong V_0$ as $\mathbf{F}_p[\mathcal{G}_R]$-modules. Denote by $\mathrm{Def}(V_0, A)$ the set of such deformations. A morphism $f: A \rightarrow A'$ in $\mathcal{C}$ induces a map $f_*: \mathrm{Def}(V_0, A) \rightarrow \mathrm{Def}(V_0, A')$ sending the class of an $A$-representation $V$ to the class of $V\otimes_{A, f}A'$. The following theorem on universal deformation ring is proved in \cite{smit}.

\begin{thm} \label{thm:6.1} \emph{(cf. \cite[Theorem 2.3]{smit})} 
There exists a universal deformation ring $A_{\mathrm{univ}} \in \mathcal{C}$ and a deformation $V_{\mathrm{univ}} \in \mathrm{Def}(V_0, A_{\mathrm{univ}})$ such that for all $A \in \mathcal{C}$, we have a bijection
\begin{equation} \label{eq:6.1}
\mathrm{Hom}_{\mathcal{C}}(A_{\mathrm{univ}}, A) \stackrel{\cong}{\rightarrow} \mathrm{Def}(V_0, A)
\end{equation}
given by $f \mapsto f_*(V_{\mathrm{univ}})$. 
\end{thm}

We deduce that when $R$ has dimension $2$ and $e$ is small, the locus of crystalline representations with Hodge-Tate weights in $[0, 1]$ cuts out a closed subscheme of $\mathrm{Spec} A_{\mathrm{univ}}$ in the following sense.

\begin{thm} \label{thm:6.2}
Suppose that $e < p-1$ and that the Krull dimension of $R$ is $2$. Then there exists a closed ideal $\mathfrak{a}_{\mathrm{BT}} \subset A_{\mathrm{univ}}$ such that the following holds: 

For any finite flat $\mathbf{Z}_p$-algebra $A$ equipped with the $p$-adic topology and any  continuous $\mathbf{Z}_p$-algebra map $f: A_{\mathrm{univ}} \rightarrow A$, the induced representation $V_{\mathrm{univ}}\otimes_{A_{\mathrm{univ}}, f} A[\frac{1}{p}]$ of $\mathcal{G}_R$ is crystalline with Hodge-Tate weights in $[0, 1]$ if and only if $f$ factors through the quotient $A_{\mathrm{univ}}/\mathfrak{a}_{\mathrm{BT}}$.
\end{thm}

\begin{proof}
This follows directly from Theorem \ref{thm:5.4} and \cite[Theorem 5.7]{moon-relativeRaynaud}.	
\end{proof}

\bibliographystyle{amsalpha}
\bibliography{library}

\providecommand{\bysame}{\leavevmode\hbox to3em{\hrulefill}\thinspace}
\providecommand{\MR}{\relax\ifhmode\unskip\space\fi MR }
\providecommand{\MRhref}[2]{%
  \href{http://www.ams.org/mathscinet-getitem?mr=#1}{#2}
}
\providecommand{\href}[2]{#2}
\begin{thebibliography}{DLLZ18}

\bibitem[And06]{andreatta}
Fabrizio Andreatta, \emph{Generalized ring of norms and generalized {$(\phi,
  \Gamma)$}-modules}, Ann. {S}ci. \'{E}c. {N}orm. {S}up\'{e}r. (4) \textbf{39}
  (2006).

\bibitem[BL95]{beuville-Laszlo}
Arnaud Beauville and Yves Laszlo, \emph{Un lemme de descente}, C.{R}. {M}ath.
  {A}cad. {S}ci. {P}aris \textbf{320} (1995), 335--340.

\bibitem[Bri08]{brinon-relative}
Olivier Brinon, \emph{Repr\'{e}sentations {$p$}-adiques cristallines et de de
  rham dans le cas relatif}, M\'{e}m. {S}oc. {M}ath. {F}r. \textbf{112} (2008).

\bibitem[BT08]{brinon-trihan}
Olivier Brinon and Fabien Trihan, \emph{Repr\'{e}sentations cristallines et
  $f$-cristaux: le cas d'un corps r\'{e}siduel imparfait}, Rend. {S}emin.
  {M}at. {U}niv. {P}adova \textbf{119} (2008), 141--171.

\bibitem[DJ95]{deJong-dieudonnemodule}
Aise~Johan De~Jong, \emph{Crystalline {D}ieudonn\'{e} module theory via formal
  and rigid geometry}, Publ. {M}ath. {I}nst. {H}autes \'{E}tudes {S}ci.
  \textbf{82} (1995), 5--96.

\bibitem[DLLZ18]{diao-lan-liu-zhu-logRH}
Hansheng Diao, Kai-wen Lan, Ruochuan Liu, and Xinwen Zhu, \emph{Logarithmic
  {R}iemann-{H}ilbert correspondences for rigid varieties}, arXiv:1803.05786,
  2018.

\bibitem[GR03]{gabber-almost}
Ofer Gabber and Lorenzo Ramero, \emph{Almost ring theory}, Lecture {N}otes in
  {M}ath., vol. 1800, Springer-Verlag, 2003.

\bibitem[Kim15]{kim-groupscheme-relative}
Wansu Kim, \emph{The relative {B}reuil-{K}isin classification of
  {$p$}-divisible groups and finite flat group schemes}, Int. {M}ath. {R}es.
  {N}ot. {I}{M}{R}{N} (2015), 8152--8232.

\bibitem[Kis06]{kisin-crystalline}
Mark Kisin, \emph{Crystalline representations and {$F$}-crystals}, Algebraic
  geometry and number theory (Boston), Progr. {M}ath., vol. 253,
  Birkh\"{a}user, 2006, pp.~459--496.

\bibitem[KL15]{kedlaya-liu-relative-padichodge}
Kiran Kedlaya and Ruochuan Liu, \emph{Relative {$p$}-adic hodge theory:
  foundations}, Ast\'{e}risque \textbf{371} (2015).

\bibitem[Moo18]{moon-relativeRaynaud}
Yong~Suk Moon, \emph{Extending $p$-divisible groups and {B}arsotti-{T}ate
  deformation ring in the relative case}, arXiv:1808.01580, 2018.

\bibitem[Sch12]{scholze-perfectoid}
Peter Scholze, \emph{Perfectoid spaces}, Publ. {M}ath. {I}nst. {H}autes
  \'{E}tudes {S}ci. \textbf{116} (2012), 245--313.

\bibitem[Sch13]{scholze-p-adic-hodge}
\bysame, \emph{$p$-adic hodge theory for rigid-analytic varieties}, Forum
  {M}ath. {P}i \textbf{1} (2013).

\bibitem[SL97]{smit}
Bart~De Smit and Hendrick~W. Lenstra, \emph{Explicit construction of universal
  deformation rings}, Modular forms and {F}ermat's last theorem (New {Y}ork),
  Springer-Verlag, 1997, pp.~313--326.

\bibitem[Vas13]{vasiu}
Adrian Vasiu, \emph{A motivic conjecture of {M}ilne}, J. {R}eine {A}ngew.
  {M}ath. \textbf{685} (2013), 181--247.

\bibitem[VZ10]{vasiu-zink-purity}
Adrian Vasiu and Thomas Zink, \emph{Purity results for {$p$}-divisible groups
  and abelian schemes over regular bases of mixed characteristic}, Doc. {M}ath.
  \textbf{15} (2010), 571--599.

\end{thebibliography}

\end{document}